\documentclass[a4paper, 11pt]{amsart}
\usepackage {a4wide, epsfig, enumerate}

\usepackage[utf8]{inputenc}  
\usepackage{lmodern}
\usepackage{amsfonts}
\usepackage{amsmath}
\usepackage{amssymb}
\usepackage{amsthm}
\usepackage{fontenc}
\usepackage{graphicx}
\usepackage{multido}         
\usepackage{pst-all}
\usepackage{pstricks}        
\usepackage{textcomp}     
\usepackage{dsfont}
\usepackage[all]{xy}
\usepackage{mathrsfs}        

\usepackage{pstricks-add} 
\usepackage{graphicx}     
\usepackage{psfrag}
\usepackage{cite}

\newcommand{\C}{{\mathbb C}}
\newcommand{\R}{{\mathbb R}}
\newcommand{\N}{{\mathbb N}}

\newcommand{\proj}{{\mathbb{P}}}
\newcommand{\Cr}{\mathcal{C}}
\newcommand{\D}{\text{D}}
\newcommand{\Case}{\text{Case }}
\newcommand{\B}{\mathbb{B}}
\newcommand{\Lr}{\mathscr{L}}
\newcommand{\A}{\mathscr{A}}
\newcommand{\F}{\mathscr{F}}

\newcommand{\supp}{{\textit{\emph{supp}}}}
\newcommand{\card}{{\textit{\emph{card}}}}
\newcommand{\id}{{\textit{\emph{id}}}}
\newcommand{\dd}{{\textit{\emph{dd}}}}

\newcommand{\e}{{\textit{\emph{e}}}}

\theoremstyle{plain}
\newtheorem{theoreme}{Theorem}[section]
\newtheorem{lemme}[theoreme]{Lemma}
\newtheorem{prop}[theoreme]{Proposition}
\newtheorem{prop/def}[theoreme]{Proposition/Definition}

\newtheorem{defi}[theoreme]{Definition}

\theoremstyle{definition}
\newtheorem{ex}[theoreme]{Example}
\newtheorem{notation}[theoreme]{Notation}

\theoremstyle{remark}
\newtheorem*{rmq}{Remark}

\usepackage[T1]{fontenc}
 \usepackage{authblk}
\author[1]{Sandrine Daurat

}
\address{Sandrine Daurat, Centre Mathématiques Laurent Schwartz, Ecole Polytechnique, Route de Saclay, 91120 Palaiseau}
\email{sandrine.daurat@math.polytechnique.fr}
\title{On the size of attractors in $\proj^k$}
\date{}
\begin{document}
\maketitle

\renewcommand{\abstractname}{Résumé}
\begin{abstract} 
Soit $f$ un endomorphisme holomorphe de $\proj^k(\C)$ possédant un ensemble attractif $\A$. Dans cet article, nous nous intéressons à la``taille" de $\A$, au sens de la géométrie complexe et de la théorie du pluripotentiel.  Nous introduisons un cadre conceptuellement simple permettant d'obtenir des ensembles attractifs non algébriques. 
Nous prouvons qu'en ajoutant une condition de dimension, ces ensembles supportent un courant positif fermé avec un quasi-potentiel borné (ce qui répond à une question de T.C. Dinh). 
Ils sont donc non pluripolaires. De plus, nous montrons que les exemples sont abondants dans $\proj^2$.
 \end{abstract}
 
\renewcommand{\abstractname}{Abstract}
\begin{abstract} 
Let $f$ be a holomorphic endomorphism of $\proj^k(\C)$ having an attracting set $\A$. In this paper, we address the question of the ``size" of $\A$ in a pluripolar sense. We introduce a conceptually simple framework to have non-algebraic attracting sets.
We prove that adding a dimensional condition, these sets support a closed positive current with
bounded quasi-potential (which answers a question from T.C. Dinh). Therefore, they are not pluripolar. Moreover, the examples are abundant on $\proj^2$.
 \end{abstract}

\section*{Introduction}

In this paper, we study the dynamics of holomorphic endomorphisms of the complex projective space $\proj^k (\C)$. Denote $\proj^k (\C)$ by $\proj^k$ from now on and let $f$ be an endomorphism of $\proj^k$ of algebraic degree $d\geq 2$. 
Such a map admits a unique invariant probability measure  $\mu$ of maximal entropy $k\log (d)$, called the equilibrium measure.  The most chaotic part of the dynamics is concentrated on $\supp (\mu)$. We refer to \cite{DiSi} for an introduction to this theme of research.
However, as opposed to dimension 1, chaotic dynamics can also occur outside $\supp (\mu)$.
A basic non trivial dynamical phenomenon outside $\supp (\mu)$ is that of attracting sets and attractors.
We refer to \cite{FSdyn, FW, JW, R, Di} for some properties and examples of attracting sets and attractors for endomorphisms of $\proj^k$ and to  \cite{Dufatou} for a basic structural description of the dynamics on the Julia set.
See also \cite{BDM} for a detailed study of a class of algebraic attractors (notice that the definition of an attractors is slightly more general there).

It is quite easy to find algebraic attracting set and attractors.
On $\proj^1$, all attractors are algebraic. Let $P,Q$ be homogeneous polynomials of degree $d>2$ in $\C^2$ with a single common zero $(0,0)$, then the line at infinity $\lbrace [z:w:t]; t=0  \rbrace$ is an attracting set for $f:[z:w:t]\mapsto [P(z,w):Q(z,w):t^d]$ and it is an attractor if the Julia set of $[z:w]\mapsto [P(z,w):Q(z,w)]$ is the whole Riemann sphere, see \cite{FSdyn}.
The first example of non algebraic attractor in $\proj^2$ (resp. $\proj^k$, with $k\geq 2$) was found by M. Jonsson and B. Weickert (resp. F. Rong), see \cite{JW} (resp. \cite{R}). 
So far, most previously known examples of non algebraic attracting sets are, in a sense, of codimension 1 and occur for maps of the form $$f:[z:w:t]\mapsto [P(z,w):Q(z,w):t^d+\varepsilon R(z,w)],$$ see \cite{FSdyn, JW, R}.
A precise notion of the dimension of an attracting set can easily be formalized, see Definition \ref{def dimension} below.

T.C. Dinh \cite{Di} constructed, under some mild assumptions, a natural positive closed current supported on the attracting set. This will be referred as \textit{the attracting current}.
 J. Taflin \cite{Ta}, under additional hypotheses, proved that this current is the unique positive closed current supported on the attracting set. It remains an interesting problem to understand the ``size" of the attracting set in a potential theoretic sense. 
For instance, Dinh \cite{Di} asks whether the quasi-potential of the attracting current is always unbounded. 
For the basic example $f:[z:w:t]\mapsto [P(z,w):Q(z,w):t^d]$, the attracting set is the line $\lbrace  t=0 \rbrace$ thus it supports a unique positive closed current of bidegree (1,1) whose quasi-potential is unbounded.

In this paper, we address this problem, by introducing the concept of a mapping of small topological degree on an attracting set (see Definition \ref{def petit dt}).
This notion was inspired by iteration theory of rational maps, see [DDG1-3].
The condition of being of small topological degree provides a conceptually simple framework to provide non algebraic attracting sets of any dimension (see Proposition \ref{prop non alg}).

In codimension 1, this condition implies that the attracting set is non pluripolar. For this we prove that it supports a positive closed current with bounded quasi-potential (see Theorem \ref{th potentiel bourne}).
This answers Dinh's question by the negative. 
On the other hand, the condition of small topological degree is not sufficient in higher codimension to assure that the attracting set is non pluripolar (see Theorem \ref{th att cont in hyperplan}). 
 
The examples of small topological maps on attracting sets are abundant. More precisely we prove the following theorem :
\begin{theoreme} \label{th intro}
Denote by $\F_d$ the  quasi-projective variety of triples $(P,Q,R)$ of homogeneous polynomials of degree $d$ in $\C^2$, such that $(0,0)$ is the single common zero of $P,Q$.
There exists a Zariski open set $\Omega\subset \F_d$ such that if $(P,Q,R)\in \Omega$ then for all $\varepsilon$ small enough, $0<|\varepsilon|<\varepsilon (P,Q,R) $, the endomorphism $f$ of $\proj ^2$ defined by 
$$f:[z:w:t]\mapsto  [P(z,w):Q(z,w):t^d + \varepsilon R(z,w)]$$
admits an attracting set $\A$ on which $f$ is of small topological degree. Moreover, $\A$ support a positive closed current $\tau$ of bidegree $(1,1)$ which admits a bounded quasi-potential.
\end{theoreme} 

Of course, this theorem cannot be true for every $(P,Q,R)\in \F_d$. For example, 
$ [z:w:t] \mapsto  [z^2 :w^2:t^2+\varepsilon z^2]$
is not of small topological degree on an attracting set, even for $\varepsilon \neq 0$.
Indeed, the attracting set is a line in this case.

Denote by  $T$ the Green current of $f$, see \cite{DiSi}.  
With the notation of Theorem \ref{th intro}, we have that  $\nu = T\wedge \tau$ is an invariant probability measure of maximal entropy supported in $\A$, see \cite{Di}. 
Under the conditions of Theorem \ref{th intro}, we infer that $\nu$ puts no mass on pluripolar sets.
\\

A recent work of N. Fakhruddin \cite{fak} gives alternate arguments for some of the results in this paper : genericity of non algebraic attracting sets, existence of Zariski dense attracting sets of higher codimension.
Notice that we do not use the same notion of genericity. 
 Fakhruddin proves that the set of holomorphic endomorphisms of $\proj^k$ which have no non trivial invariant set (i.e. not of zero dimension or not $\proj^k$) contains a countable intersection of Zariski open sets.
On the other hand, we work in a specific family  $\F_d$ and construct a Zariski open set of examples there. 
We also observe that the set of endomorphisms in $\proj^2$
possessing an attracting set of small topological degree is open for the usual
topology.
\\

The outline of the paper is as follows.
We start with some preliminaries.
Then we introduce the notion of small topological degree on an attracting set. We show that it is a sufficient condition to have a non algebraic attracting set and under this hypothesis, in codimension one, the attracting set is non-pluripolar.

Section 3 is devoted to the proof of Theorem \ref{th intro}.

In section 4, we will exhibit examples in $\proj^3$ of attracting sets of codimension 2. We will see, that the condition of small topological degree is not sufficient in higher codimension to assure that the attracting set is non pluripolar. 
Using Hénon-like maps of small topological degree, we will exhibit the first explicit example of a Zariski dense attracting set of codimension 2.

We finish with some remarks around Theorem \ref{th intro} and some open questions.

\section{Preliminaries}

\subsection{Attracting sets and attractors}
In this section we give some definitions and recall the framework of \cite{Di}.

\begin{defi}\label{def attracteur}
A set $\A$ is called an \textbf{attracting set} if there exists an open set $U$ such that
 $f(U)\Subset U$ and $\A=\cap_{n\geq 0} f^n(U)$. We call such an open set $U$ a \textbf{trapping region}.
If moreover $f$ is topologically mixing on $\A$ then $\A$ is called an \textbf{attractor}.
\end{defi}

The following proposition is obvious.

\begin{prop}
Let $\A$ be an attracting set. Then $\A$ is closed and $f(\A)=\A$.
\end{prop}

\begin{defi} \label{def dimension}
An attracting set in $\proj^k$ is said to be of \textbf{dimension} $k-m$ if it supports a positive closed current of bidimension $(k-m,k-m)$ (i.e. of bidegree $(m,m)$) but none of bidimension $(l,l)$ with $l>k-m$.
\end{defi}
~

\begin{rmq}
~
\begin{itemize}
\item Equivalently, an attracting set $\A \subset \proj^k$ is said to be of dimension $k-m$ if  any (or one) trapping region of $\A$ supports a (smooth) positive closed current of bidimension $(k-m,k-m)$ but none of bidimension $(l,l)$ with $l>k-m$. 

\item If $\A$ is of dimension $k-m$ then the Hausdorff dimension of $\A$ is at least $2(k-m)$.
\end{itemize}
\end{rmq}

The proof of the following elementary proposition is left to the reader. 
\begin{prop}\label{prop dimension}
If $\A$ is an algebraic attracting set of dimension $k-m$ then Definition \ref{def dimension} is equivalent to the classical definition of dimension for algebraic sets.
\end{prop}

\begin{ex}
Let $f$ be an endomorphism of $\proj^k$ and let $U$ be an open set such that $f(U)\Subset U$.
We assume that there exists two projective subspaces $I$ and $J$ of $\proj^k$ of dimension $m-1$ and $k-m$, $1\leq m \leq k-1$, such that $J\subset U$ and  $I\cap U=\emptyset$.
Then $\A=\bigcap f^n(U)$ is of dimension $k-m$. In fact, $U$ supports the positive closed current $[J]$ of bidimension $(k-m,k-m)$ but does not support a positive closed current of bidimension $(l,l)$ with $l>k-m$.
Indeed, the support of any positive closed current $T$ of bidimension $(l,l)$ with $l>k-m$ intersects  $I$. \qed
\end{ex}

Before stating the main result of \cite{Di}, we need some notation. Let $\pi : \proj^k \setminus I \rightarrow J$ be the projection of center $I$.
More precisely,
letting $I(x)$ be the projective space, of dimension $m$, containing $I$ and passing through a point $x\in \proj^k \setminus I$, then $\pi(x)$ is the 
unique intersection point between $J$ and $I(x)$. We consider the point $\pi(x)$ as the origin of the complex vector space $I(x)\setminus I\simeq \C^m$,
where $I$ is viewed as the hyperplane at infinity of $I(x)\simeq \proj^m$. In other words, $\proj^k \setminus I$ is viewed as a vector bundle over 
$J$. If $x\in J$, we have that $\pi(x)=x$. We suppose that
\begin{equation}\label{cond dinh}
 \text{the open set } U\cap I(x) \text{ in } I(x)\setminus I\simeq \C^m \text{ is star-shaped at } x \text{ for every } x\in J.
\end{equation}
If $I$ is a point and $J$ is a projective hyperplane, i.e. $m=1$, the previous hypothesis is equivalent to the property that the open subset $\proj^k\setminus U$ is star-shaped at $I$.

\begin{theoreme}[Dinh]\label{th dinh} 
Let $f$ and $U$ be as above. Let $S$ be a  positive closed $(m,m)$-form of mass 1 with continuous coefficient  and with compact support in $U$.
 Then the sequence $\left( \frac{1}{d^{k-m}}(f^n)_* S\right)$ converges to a positive closed current $\tau$ of mass 1, with support in $\A=\bigcap f^n(U)$.
The current $\tau$ does not depend on $S$.
Moreover, it is woven,
 forward invariant (i.e. under $\frac{1}{d^{k-m}}f_*$) and is extremal in the cone of invariant positive closed currents of bidegree $(m,m)$ with support in $\A$.
\end{theoreme}

In the sequel we will refer to $\tau$ as \textit{the attracting current} of $\A$.

\subsection{Pluri-potential theory}

Here we focus on the codimension 1 case. Let $f$ be an endomorphism of $\proj^k$ and $\Lr$ be the normalised push forward operator, i.e. $\Lr=\frac{1}{d^{k-1}} f_*$.

Let $T$ be a positive closed current  of bidegree $(1,1)$ of mass 1.
 There exists a quasi-psh function $u$, i.e. $u$ is locally the difference of a plurisubharmonic function (psh for short) and a smooth function, such that $T-\omega_{FS}=dd^c u$.
We call such a function $u$ a quasi-potential of $T$. 

A dsh function is the difference of two quasi-psh functions, see Appendix A.4. of \cite{DiSi} for  properties of dsh functions and \cite{De} for basics on psh function.
Let $R,S$ be two positive closed currents on $\proj^k$ of bidegree (1,1) with the same mass then there exists a dsh function $w$ such that $R-S=\dd^c w$.

\begin{notation} 
Let $S,R$ be positive closed currents of bidegree (1,1) with mass 1. We denote by $u_{S,R}$ the unique dsh function  such that $S-R=\dd^c u_{S,R}$ and $\int_{\proj^k} u_{S,R} \, \omega_{FS}^k=0$.
\end{notation}

\begin{lemme}\label{lemme CV pot}
If $(S_i,R_i)$ converges, in the sense of currents, toward $(S,R)$ then $(u_{S_i,R_i})$ converges, in $L^1$ and
 in sense of distributions, toward $u_{S,R}$. 
\end{lemme}

\begin{proof}
By Thoerem A.40 of \cite{DiSi}, $u_{S,\omega_{FS}}$ depends continuously on $S$. As $u_{S,R} = u_{S,\omega_{FS}} - u_{R,\omega_{FS}}$ the proof is complete.
\end{proof}

\begin{lemme}\label{lemme forward pot}
Let $R,S$ be as above, then $\Lr \dd^c u_{S,R} = \dd^c  \frac{1}{d^{k-1}} f_* u_{S,R}$ and
 $\frac{1}{d^{k-1}} f_* u_{S,R}(x)=\frac{1}{d^{k-1}}\underset{f(y)=x}{\sum} u_{S,R}(y)$.
\end{lemme}

\begin{rmq}
In the sum $\underset{f(y)=x}{\sum} u_{S,R}(y)$, the preimages are counted with multiplicity. As $f$ is a 
holomorphic endomorphism, if $u$ is continuous map then $x\mapsto\underset{f(y)=x}{\sum} u(y)$ is 
continuous.
\end{rmq}

\begin{proof}
We just need to prove this locally in $\proj^k$.
 As $f$ is finite, if $V$ is a small enough open set there exists a psh function $u$, defined on $U=f^{-1}(V)$, such that $T_{|U}=dd^c u$. 

\begin{lemme}
With the previous notation, $x\mapsto v(x)=\frac{1}{d^{k-1}}\underset{f(y)=x}{\sum} u (y)$
is psh on $V$ and $\left(\Lr T\right)|_V=\dd^c v$. 
\end{lemme}

This is classical. We recall the proof for completeness.

\begin{proof}
Up to taking a decreasing regularisation, we may assume that $u$ is smooth.
Denote by $C_f$ the critical set of $f$.
As  $f_*dd^c u=dd^c f_* u$, \cite[Ch.1 Theorem (2.14)]{De}, and $f$ is a submersion on $U\setminus C_f$, we have that 
\begin{equation}\label{eq pot}
f_* u(x)=\underset{f(y)=x}{\sum} u (y) \text{ on } V\setminus f(C_f),
\end{equation}
 see \cite[Ch.1 (2.15)]{De}. Outside $C_f$, the map $f$ is locally a biholomorphism so $f_* u$ is psh on $V\setminus f(C_f)$.
 Moreover, $f(C_f)$ is pluripolar and $x\mapsto\underset{f(y)=x}{\sum} u(y)$ is continuous on $V$.
Recall that a locally bounded psh function on $V\setminus f(C_f)$  admits a unique psh extension on $V$. 
Thus the equation  \eqref{eq pot} is true for all $x\in V$.
\end{proof}

We now finish the proof of Lemma \ref{lemme forward pot}. There exists two psh functions $u,v$ on $U$ such that $R_{|U}=\dd^c u$ and $S_{|U}=\dd^c v$.
So, on $U$, we have that $\dd^c(u-v)=\dd^c (u_{S,R})$, thus there exists a pluri-harmonic function $h$ such that $u+h-v=u_{S,R}$.
 Then $u+h$ and $v$ are psh functions such that $R_{|U}=\dd^c (u+h)$ and $S_{|U}=\dd^c v$. By the preceding lemma, we have that $\frac{1}{d^{k-1}}f_*R-\frac{1}{d^{k-1}}f_*S=\dd^c \frac{1}{d^{k-1}}f_*u_{S,R} $ on $f(U)$, with $\frac{1}{d^{k-1}}f_*u_{S,R}(x)=\frac{1}{d^{k-1}}\underset{f(y)=x}{\sum} u_{S,R} (y)$.
\end{proof}

\section{Mappings of small topological degree  on an attracting set}

In this section, we introduce the notion of being of small topological degree on an attracting set in $\proj^k$,
 prove that such an attracting set is never algebraic and that, in codimension 1, it is non pluripolar.

\begin{defi}
Let $f$ be an endomorphism of $\proj ^k$ of algebraic degree $d$ and let $U$ be a trapping region for an attracting set $\A$ of dimension $k-m$. The endomorphism $f$ is said to be of small topological degree on $U$ 
if the number of preimages in $U$ of any point belonging to $f(U)$ is strictly less than $d^{k-m}$.
The endomorphism $f$ is said to be \textbf{asymptotically of small topological degree} on $U$ if for all $p\in f(U)$ we have that $\overline{\lim} \left(\card (f^{-n}(p)\cap U)\right)^{1/n} <d^{k-m}$.
\end{defi}

\begin{rmq}
The notion of being of small topological degree depends on the choice of $U$. We will encounter in Section \ref{section exemples} examples where $f^3$ is of small topological degree on $f(U)$ but not on $U$.
\end{rmq}

The proof of the following proposition is left to the reader.

\begin{prop}
The definition of asymptotic small topological degree does not depend on the choice of the trapping region. 

Moreover, if $f$ is asymptotically of small topological degree  then for each trapping region $U$ there exists $n\geq 1$ such that $f^n$ is of small topological degree on $U$.
\end{prop}

\begin{defi}\label{def petit dt}
We say that $f$ is of \textbf{small topological degree} on an attracting set if $f$ is asymptotically of small topological degree on some trapping region.

Sometimes, we will abbreviate this into ``attracting set of small topological degree".
\end{defi}

\begin{prop}\label{cond ouverte}
The property of being of small topological degree on some attracting set is open in the set of endomorphisms of degree $d$.
\end{prop}

\begin{proof}
Let $f$ be an endomorphism of $\proj ^k$ of small topological degree on some attracting set of dimension $k-m$ and let $U$ be a trapping region. Replacing $f$ by an iterate we may assume that $f$ is of small topological degree on $U$.
Then for each $p\in f(U)$, we have that $\card (f^{-1}(p) \cap U )<d^{k-m}$.
Let $g$ be close to $f$, there exists an open set $V$ such that $g(V)\Subset f(U)\Subset V\Subset U$.
Thus, the attracting set $\cap g^n(V)$ is of dimension $k-m$. For each $p\in g(V)\subset U$, $g^{-1}(p)$ is close to $f^{-1}(p)$ hence, if $g$ is close enough to $f$, we have that $\card (g^{-1}(p) \cap V )\leq \card (f^{-1}(p) \cap U ) <d^{k-m}$.
Therefore, $g$ is of small topological degree on $V$.
\end{proof}

Let us prove that such attracting sets are non algebraic.

\begin{prop}\label{prop non alg}
If a holomorphic endomorphism $f$ of $\proj^k$ is of small topological degree on an attracting set $\A$, then $\A$ is non algebraic.
\end{prop}

The proof relies on the following classical lemma.

\begin{lemme}\label{intersection transversale}
Let $f$ be an endomorphism of $\proj ^k$ and $M\subset \proj^k$ be an invariant algebraic set ($f(M)=M$) of pure dimension $k-m$, $1\leq m\leq k-1$. Then there exists $n\geq 1$ such that $f^n$ fixed all irreducible components of $M$ and the topological degree of $f^n_{|M}$ is $d^{n(k-m)}$.
\end{lemme}

\begin{proof}
See \cite[Lemma 1.48]{DiSi}.
\end{proof}

\begin{proof}[Proof of Proposition \ref{prop non alg}]
Let $k-m$ be the dimension of $\A$ and let $U$ be a trapping region. Up to replacing $f$ by an iterate, we may  assume that $f$ is of small topological degree on $U$.
Assume that $\A$ is algebraic. By Proposition \ref{prop dimension}, $\A$ is of dimension $k-m$. Let  $M$ be the (finite) union of irreducible components of pure dimension $k-m$ of $\A$. 
Then $f(M)=M$, because $f(\A)=\A$ and $f$ does not contract any algebraic subvariety on an algebraic subvariety  of lower dimension. 
Hence, by Lemma \ref{intersection transversale}, $f$ cannot be of small topological degree on $M$.
\end{proof}


We now proceed with the non pluripolarity (in codimension 1).
The corresponding result in higher codimension fails (see below Theorem \ref{th att cont in hyperplan}).

\begin{theoreme}\label{th potentiel bourne}
Let $f$ be a holomorphic endomorphism of $\proj^k$ and let $\A$ be an attracting set of dimension $k-1$.
Assume $f$ is of small topological degree on some trapping region $U$ of $\A$. Let $T$ be a closed positive current of bidegree (1,1) of mass 1, with support in $U$, admitting a bounded quasi-potential. Then each cluster value of $\left( \frac{1}{d^{n(k-1)}} (f^n)_* T \right)$ has a bounded quasi-potential. In particular, $\A$ is non pluripolar.

Moreover, under the assumptions of Theorem \ref{th dinh}, the attracting current $\tau$ has a bounded quasi-potential.
\end{theoreme}

The proof of this theorem is essentially contained in the following lemma.
We recall that $\Lr=\frac{1}{d^{k-1}} f_*$. We denote by $d_t$ the maximum number of preimages in $U$ of a point in $f(U)$ under $f$.

\begin{lemme}\label{lemme potenitel bourne}
Let $S,R$ be positive closed currents of bidegree (1,1) with mass 1 and with bounded quasi-potentials.
For every $U'$ such that $f(U)\Subset U'\Subset U$, there exists a constant $c>0$ such that if  $\supp(S)\subset U'$ and $\supp(R)\subset f(U')\subset U'$ then we have that $$||u_{\Lr S,R}||_\infty \leq \frac{d_t}{d^{k-1}}||u_{S,R}||_\infty +||u_{\Lr R,R}||_\infty +c  .$$
\end{lemme}

\begin{proof}
As expected, the proof just consists in making precise the idea that since $u_{S,R}$ is pluri-harmonic on $\proj^k\setminus U'$,
we do not need to focus on preimages belonging to $\proj^k\setminus U'$. 
Assume that we can separate the preimages inside $U$ and outside $U$ such that
  the map $v$ defined by :
$$v(x)=\frac{1}{d^{k-1}}\underset{y\in U}{\underset{f(y)=x}{\sum}} u_{S,R} (y),$$
is dsh and satisfies $\dd^c v= \dd^c(\frac{1}{d^{k-1}} f_* u_{S,R})$.
Then we have $||v||_\infty\leq \frac{d_t}{d^{k-1}}||u_{S,R}||_\infty$ and we would finish the proof simply by showing that $u_{\Lr S,R}-v-u_{\Lr R,R}$ is bounded.
However $f:U\rightarrow f(U)$ is not proper, so we actually need to work locally near every point to make this idea work.

Let $V$ be a small open sets such that $f^{-1}(V)$ may be written as a disjoint union $f^{-1}(V) = U_1\cup U_2$ such that $f(U_1)=V$,
 $U_1\subset U$, $U_2\cap U'=\emptyset$ and
 the number of preimages (with multiplicity) of a point of $V$ in $U_{1}$ is fixed (and less than $d_t$).

In this way, we have $$\frac{1}{d^{k-1}}\underset{{f(y)=x}}{\sum} u_{S,R} (y)= v_1(x)+v_2(x)$$ and $\dd^c v_1=\dd^c(\frac{1}{d^{k-1}} f_* u_{S,R})$, where
\begin{eqnarray}
\nonumber
v_1(x)=\frac{1}{d^{k-1}}\underset{y\in U_1}{\underset{f(y)=x}{\sum}} u_{S,R} (y) \text{ and }
v_2(x)=\frac{1}{d^{k-1}}\underset{y\in U_2}{\underset{f(y)=x}{\sum}} u_{S,R} (y).
\end{eqnarray}

As $f$ is finite, we may choose for any $p\in f(U)$ a neighbourhood $V_p$ small enough such that  each connected component of $f^{-1}(V_p)$ contains a unique preimage of $p$. (Of course $V_p$ cannot be uniform with respect to $p$.)
Denote by $C_q$ the  connected component of $f^{-1}(V_p)$ which contains $q\in f^{-1}(p)$.
 We denote by $U_{1,p}$ the set $U_{1,p}=\cup _{q\in f^{-1}(p)\cap U} C_q$ and by $U_{2,p}$ the set$U_{2,p}= f^{-1}(V_p)\setminus U_{1,p}$. 
We can reduce $V_p$ such that $U_{1,p}\subset U$, $f(U_{1,p})=V_p$ and $U_{2,p} \cap U' = \emptyset$, 
thus the number of preimages (with multiplicity) of a point of $V_p$ which lies in $U_{1,p}$ is fixed.

Let $x$ belonging to $V_p$. By Lemma \ref{lemme forward pot} $ x\mapsto \widetilde{u}_{S,R}(x)= \frac{1}{d^{k-1}} \sum_{f(y)=x} u_{S,R} (y)$ is a dsh function such that $\Lr (S-R) =\dd^c \widetilde{u}_{S,R}$.
We define,  $v_{S,R,p} $ by : for all $x\in V_p$
 $$v_{S,R,p} (x)=  \frac{1}{d^{k-1}}\underset{y\in U_{1,p}}{\underset{f(y)=x}{\sum}} u_{S,R} (y) .$$

As $\supp(S-R)\subset U'$, $u_{S,R}$ is pluri-harmonic outside $U'$ 
thus $\dd^c v_{S,R,p} =\dd^c  \widetilde{u}_{S,R}=\Lr (S-R)$ in $V_p$ and 
$||v_{S,R,p}||_\infty \leq \frac{d_t}{d^{k-1}} ||u_{S,R}||_\infty$.
As $\dd^c u_{\Lr S,R}=\Lr S - R=\Lr (S-R) + (\Lr R-R) =\dd^c (v_{S,R,p}+u_{\Lr R,R})$ on $V_p$, we just have to prove that  there exists $c$ independent of $S,R$ and $V_p$ such that
on $V_p$ $||u_{\Lr S,R} - (v_{S,R,p}+u_{\Lr R,R})||_\infty \leq c.$

We know that $u_{\Lr S,R} - (v_{S,R,p}+u_{\Lr R,R})$ is pluri-harmonic on $V_p$, hence, reducing a little bit $V_p$ if necessary, by the maximum principal there exists $c_{S,R}$ such that 
$||u_{\Lr S,R} - (v_{S,R,p}+u_{\Lr R,R})||_\infty \leq c_{S,R}$ on $V_p$.
Let us first show that $c_{S,R}$ is uniform with $R$ and $S$.
 Assume by contradiction that there exists a sequence $(S_i,R_i)$ such that $||u_{\Lr S_i,R_i} - (v_{S_i,R_i}+u_{\Lr R_i,R_i})||_\infty\rightarrow \infty$.
We can extract a converging subsequence, still denoted by $(S_i,R_i)$. Let $(\tilde{S},\tilde{R})$ be its limit.
By Lemma \ref{lemme CV pot}, the sequence of pluri-harmonic functions
$\left( u_{\Lr S_i,R_i} - (v_{S_i,R_i}+u_{\Lr R_i,R_i})\right)$ converges, in sense of distributions, toward 
$u_{\Lr \tilde{S},\tilde{R}} - (v_{\tilde{S},\tilde{R}}+u_{\Lr \tilde{R},\tilde{R}})$, which is bounded, therefore the convergence is uniform (see corollary 3.1.4 in \cite{H}).
Thus the sequence $(||u_{\Lr S_i,R_i} - (v_{S_i,R_i}+u_{\Lr R_i,R_i})||_\infty )$ is bounded, which contradicts the hypothesis.
Therefore, there exists $c_1$ independent of $(S,R)$ such that  $||u_{\Lr S,R} - (v_{S,R}+u_{\Lr R,R})||_\infty \leq c_1$ hence
$$||u_{\Lr S,R}||_\infty \leq \frac{d_t}{d^{k-1}} ||u_{S,R}||_\infty + ||u_{\Lr R,R}||_\infty +c_1 \text{ on } V_p .$$

Up to shrinking $U$, we may assume that $f(U)$ is covered by a finite number of such neighbourhoods $V_p$.
We infer that there exists $c_2$ such that $||u_{\Lr S,R}||_\infty \leq \frac{d_t}{d^{k-1}} ||u_{S,R}||_\infty + ||u_{\Lr R,R}||_\infty +c_2$ on $f(U)$.
Finally, as $u_{\Lr S,R}$ is pluri-harmonic outside $f(U')$ and $f(U')\Subset f(U)$,  $u_{\Lr S,R}$ is bounded by a constant $c_{U}$ on $\proj^k\setminus f(U)$.
Thus, we obtain that
 $$||u_{\Lr S,R}||_\infty \leq \frac{d_t}{d^{k-1}} ||u_{S,R}||_\infty + ||u_{\Lr R,R}||_\infty + c_2 +c_{U}\text{ on } \proj^k.$$
For similar reasons as in the case of $c_{S,R}$, we may assume that $c_{U}$ is independent of $S,R$. 
 Therefore, the proof of the lemma is complete.
\end{proof}

\begin{proof}[Proof of Theorem \ref{th potentiel bourne}]
Up to replacing $T$ by $\Lr T$, we assume $\supp(T)\subset f(U')\subset U'$.
Denote by $u_n$ the dsh function such that $\Lr^n T-T=\dd ^c u_n$ and $\int u_n \, \omega^k_{FS}=0$, i.e. $u_n=u_{\Lr^n T,T}$. Then $u_1$ is bounded.
By Lemma \ref{lemme potenitel bourne}, with $S=\Lr^n T$ and $R=T$, we have
$$ ||u_{n+1}||_\infty \leq \dfrac{d_t}{d^{k-1}} ||u_n||_\infty +||u_1||_\infty +c \text{ on }U.$$
It is classical that a sequence $(x_n)$ satisfying $0\leq x_{n+1}\leq \alpha x_n +c$ with $0< \alpha <1$ is bounded.
Therefore, each cluster value of $(\Lr ^n T)$ has a bounded quasi-potential. If in addition, $f$ satisfies \eqref{cond dinh} then $(\Lr ^n T)$
converges toward $\tau$, thus $\tau$ has also a bounded quasi-potential.
\end{proof}

\begin{rmq}
I do not know how to prove the continuity of the quasi-potential. A reason for this is that the proof does not yield the convergence of the sequence of potentials. Note that the same difficulty appears in Dinh's Theorem \ref{th dinh}.
\end{rmq}

\section{A class of attracting sets of small topological degree in $\proj^2$}\label{section exemples}

Let $\F_d$ be the set of triples $(P,Q,R)$ of homogeneous polynomials of degree $d\geq 2$ in $\C^2$, such that $(0,0)$ is the single common zero of $P,Q$.
Clearly, $\F_d$ is a quasi-projective variety.
In this section, we consider a particular class of endomorphisms of $\proj ^2$ given by the formula
\begin{equation}\label{def f}
f:[z:w:t]\rightarrow  [P(z,w):Q(z,w):t^d + \varepsilon R(z,w)]
\end{equation}
where $(P,Q,R) \in \F_d$.
Attractors for mappings of this form were studied by J.E. Forn\ae ss and N. Sibony  in \cite{FSdyn}, as well as F. Rong  \cite{R}.

These endomorphisms preserve the pencil of lines passing through $[0:0:1]$.
Such a line is determined by a point of the form $[z:w:0]$, thus we identify the line passing through $[0:0:1]$ and $[z:w:0]$ 
with the point of the line at infinity $[z:w:0]$.
The dynamics on this pencil of lines is that of $f_{\infty}:[z:w]\rightarrow [P(z,w):Q(z,w)]$ on $\proj ^1$.

\begin{theoreme}\label{theoreme}
There exists a Zariski open set $\Omega \subset \F_d$ such that if  $(P,Q,R)\in \Omega$ then for  small enough $\varepsilon$, $0<|\varepsilon|<\varepsilon (P,Q,R) $, the endomorphism $f$ of $\proj ^2$ defined by 
$$f:[z:w:t]\mapsto  [P(z,w):Q(z,w):t^d + \varepsilon R(z,w)]$$
admits an attracting set $\A$ of small topological degree. In particular, $\A$ is non pluripolar. Moreover, $f$ satisfies the condition of Theorem \ref{th dinh} and the attracting current $\tau$ has a bounded quasi-potential.
\end{theoreme}

\begin{rmq}
Since having an attracting set of small topological degree is an open condition, by perturbing in the set of all endomorphisms of $\proj^2$, we obtain maps which do not preserve a pencil but also satisfy the conclusion of Theorem \ref{theoreme}.
\end{rmq}

The proof gives specific examples in any degree and we can arrange for $f$ to be topologically mixing on $\A$,  therefore we obtain the following :

\begin{prop}\label{prop exist attracteur non pp}
There exists  an \emph{attractor} $\A\subset \proj^2$  of small topological degree.
\end{prop}

We now proceed to the proof of Theorem \ref{theoreme}.

\begin{proof}[Proof of theorem \ref{theoreme}]
The proof will proceed in four steps. First we show that any $f$ of the form \eqref{def f} admits an attracting set and satisfies the conditions of Theorem \ref{th dinh}. Next we give a sufficient condition ensuring that $f^3$ is of small topological degree on some trapping region. The third step is to show that this condition is algebraic. Finally, we give examples in any degree, therefore showing that the condition is generically satisfied.
\\

\textbf{Step 1.} The map $f$ admits an attracting set and satisfies the conditions of Theorem \ref{th dinh}.
\\

As $R$ is a homogeneous polynomial of degree $d$, there exists $\beta >0$ such that for all $(z,w)\in\C^2 $ we have that $|R(z,w)|\leq \beta \max(|z|,|w|)^d$. 
Up to multiplying $\varepsilon$ by $\beta$ we may assume that $\beta =1$.
Denote by $\alpha$ the constant $\alpha = \underset{\max(|z|,|w|)=1}{\inf} \max (|P(z,w)|,|Q(z,w)|)$, as $(0,0)$ is the single common zero of $P,Q$, we have that $\alpha >0$.

Let $U_{\rho}$ be the open set of  $[z:w:t]\in \proj^2$ such that $|t|<\rho \max(|z|,|w|)$.
Then we have
$|t^d+\varepsilon R(z,w)|<(\rho^d+\varepsilon) \max(|z|,|w|)^d <\frac{1}{\alpha}(\rho^d+\varepsilon) \max(|P(z,w)|,|Q(z,w)|)$, by definition of  $\alpha$.
To assume that $f(U_{\rho}) \Subset U_{\rho}$, all we need is that $\frac{1}{\alpha}(\rho^d+\varepsilon)<\rho$. 
We choose $\varepsilon \ll 1$. We let the reader check that the choice 
\begin{equation}\label{choix rho}
\rho = 4\frac{\varepsilon}{\alpha}
\end{equation}
is convenient.

Thus $\A= \cap_{n\in \N} f^n (U_\rho)$ is an attracting set.
Denote by $l_{[z:w]}$ the line passing through $[0:0:1]$ and $[z:w:0]$ and by $L_\infty$ the line at infinity.
The line $l_{[z:w]}$ intersects $U_\rho$ into an open disc centered at $[z:w:0]$. Thus, by taking $I=\lbrace [0:0:1] \rbrace$, $J=L_\infty$ and $U=U_\rho$, $f$ satisfies the hypothesis \eqref{cond dinh}.
\\

\textbf{Step 2.} A sufficient condition for being of small topological degree.
\\

Here we will introduce for $(P,Q,R)\in\F_d$ subsets $X$ and $Y$ of $L_\infty \simeq \proj^1$ and a condition on these sets insuring that for all $\varepsilon\neq 0$ small enough $f^3$ is of small topological degree on $f(U_\rho)$, where $f$ is defined by \eqref{def f}, (see Proposition \ref{prop petit deg top}).

Throughout the proof, by a line we mean a line passing through $[0:0:1]$, unless for the line at infinity $L_\infty$. 
Recall that the image and the preimages of a line is a line and that $f(l_{[z:w]})=l_{f_\infty({[z:w]})}$. For the sake of convenience, we will confuse the line $l_{[z:w]}$ with the point $[z:w]$ in $L_\infty$ and denote by $f_\infty$ the action of the lines. So $X$ and $Y$ can be seen as a set of points in $L_\infty$ or a set of lines.

The attracting set we obtain is a complex version of the solenoid.
The difference is that the ``branches'' of $f(U_\rho)$ must necessarily cross, so the map cannot be injective. 
The proof consists in analysing the geometry of $f(U_\rho)$, as well as its self-crossings, and the behaviour of the preimages, such that estimating the number of preimages staying in $f(U_\rho)$ reduces to a combinatoric problem.

For maps of the form \eqref{def f}, a generic line has $d$ preimages 
and a generic point has $d^2$ preimages (in $\proj^2$), with $d$ in each line.
So there are two ways to control the number of preimages (which lie in $f(U_\rho)$) of a point :
\begin{itemize}
\item controlling the number of lines that contain preimages staying in $U_\rho$.
\item For each points in such a line, bounding the number of preimages that belong to $f(U_\rho)$.
\end{itemize}
We leave the reader check that the image of a disc in $l_{[z:w]}$ centered at $[z:w:0]$ is a disc in $l_{f_\infty ([z:w])}$ centered at $f([z:w:0])$.
A line $l$ having  $d$ preimages (with multiplicity), $f(U_\rho)\cap l$ is made of at most $d$ discs which might intersect.
\def\svgwidth{2cm}
\begin{figure}[!h]
\begin{center}
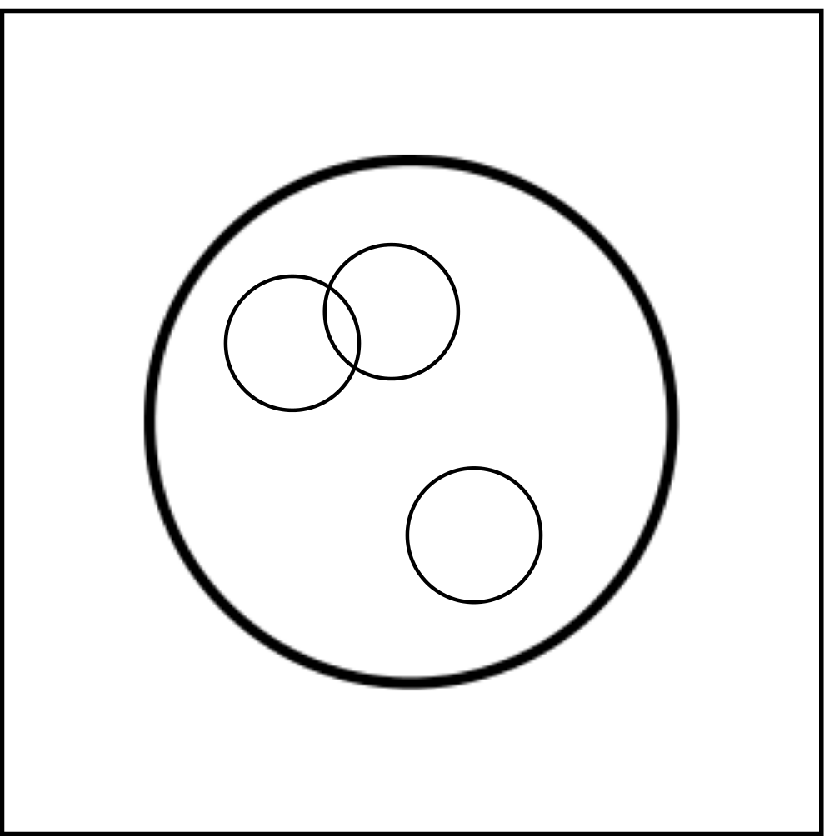
\caption{An example of the intersection of a line with $U_{\rho}$ (large circle) and $f(U_{\rho})$ (small circle). Here $f$ is of degree 3.}
\end{center}
\end{figure}

Denote by $A$ the set of lines $l$ such that
 $f(U_\rho)\cap l$ is made of $d$  disjoint discs.
Thus if $l\in A$,  the preimages of a point $p\in f(U_\rho)\cap l$, 
which lie in $U_\rho$, are contained in a single line $l_{-1}$, i.e. $f^{-1}(p)\cap U_\rho\subset l_{-1}$.
 Through the end of the proof, for each $i\in \lbrace 1,..,d\rbrace$,  we normalize $(z_i,w_i)$ so that  $$P(z_i,w_i)=z \text{ and } Q(z_i,w_i)=w.$$
 
The line $l_{[z:w]}$ is in $A$ iff for  $i\neq j$ in $\lbrace 1,..,d \rbrace$ and $t_i,t_j$ such that $[z_i:w_i:t_i],[z_j:w_j:t_j]\in U_\rho$ we have  that $t_i^d + \varepsilon R(z_i,w_i) \neq  t_j^d + \varepsilon R(z_j,w_j)$.
By definition of $U_\rho$, for this, it suffices that
\begin{equation}\label{cond A0}
|\varepsilon R(z_i,w_i) -  \varepsilon R(z_j,w_j) |>  \rho^d (\max (|z_i|,|w_i|)+\max (|z_j|,|w_j|))^d.
\end{equation}
\def\svgwidth{10cm}
\begin{figure}[!h]
\begin{center}
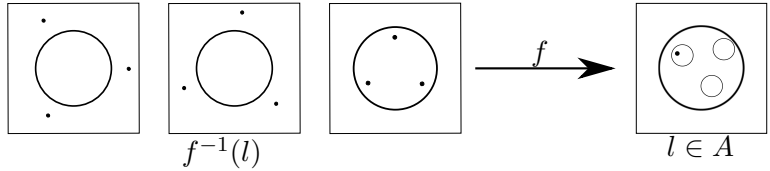
\caption{The intersection of a line $l\in A$ with $U_{\rho}$ and $f(U_{\rho})$, on the right, and the preimages of $l$ and of a point in $f(U_{\rho})\cap l$, on the left.}
\end{center}
\end{figure}

The points $[z:w: \e^{\frac{2ik\pi}{d}}t ]$ all have the same image and are either all in $U_\rho$, or all outside $U_\rho$. 
Therefore, for any line  $l$ and any point $p\in l\cap f(U_\rho)$, if we fix a preimage $l_{[z:w]}$ of $l$, the point $p$ has $d$ preimages (with multiplicity) in the line $l_{[z:w]}$, they are of the form $[z:w: \e^{\frac{2ik\pi}{d}}t ]$. 
Thus,  $p$ has 0 or $d$ preimages (with multiplicity) in $l_{[z:w]}\cap U_\rho$.
To further reduce the number of preimages we must see if they are or not in $f(U_\rho)$.

Denote  by $B_{-1}$ the set of lines $l_{[z:w]}$ such that for all $|t|<\rho \max(|z|,|w|)$ there exists at most one $k\in \lbrace 0,..,d-1\rbrace$
such that $[z:w: \e^{\frac{2ik\pi}{d}}t ] \in f(U_\rho)$ and denote by $B$ the set of lines whose preimages all of lie in  $B_{-1}$.

The line  $l_{[z:w]}$ is in $B_{-1}$ iff for all $i,j\in \lbrace 1,..,d \rbrace$, all $t_i,t_j$ such that $[z_i:w_i:t_i],[z_j:w_j:t_j]\in U_\rho$ and all $k\in \lbrace 1,..,d-1 \rbrace$ we have that $t_i^d + \varepsilon R(z_i,w_i) \neq \e^{\frac{2ik\pi}{d}} (t_j^d + \varepsilon R(z_j,w_j))$.
For this, it suffices that for each $k\in \lbrace 1,..,d-1 \rbrace$
\begin{equation}\label{cond A1}
|\varepsilon R(z_i,w_i) -  \e^{\frac{2ik\pi}{d}}\varepsilon R(z_j,w_j) |>  \rho^d (\max (|z_i|,|w_i|)+\max (|z_j|,|w_j|))^d.
\end{equation}

 \def\svgwidth{6cm}
\begin{figure}[!h]
\begin{center}
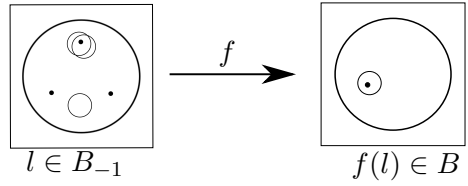
\caption{The intersection of the lines $l$ and $f(l)$ with $U_{\rho}$ and $f(U_{\rho})$, and, the preimages of a point in $f(l)\cap f(U_{\rho})$, when $l\in B_{-1}$.}
\end{center}
\end{figure}

In this way, if $l \in B$ and $p\in l\cap f^2(U_\rho)$, for any preimage $l_{-1}$ of $l$,
$p$ has at most one preimage in $l_{-1} \cap f(U_\rho)$.
If in addition $l\in B \cap A$ then a point $p\in l\cap f^2(U_\rho)$ has a single preimage in
$f(U_\rho)$.
\\

We need to understand the complement of $A \cap B$.
\\

Denote by $l_{[z_1:w_1]},..,l_{[z_d:w_d]}$ the preimages (with possible repetitions)  of a line $l_{[z:w]}$.

A sufficient condition for a line $l_{[z:w]}$ not to be in $A$ is that there exists  $i\neq j$ in $\lbrace 1,..,d \rbrace$  such that $f([z_i:w_i:0]) = f([z_j:w_j:0])$.
Denote by $X_{-1}$ the set of lines $l_{[z_i:w_i]}, l_{[z_j:w_j]}$ which satisfy this condition and let $X=f_\infty(X_{-1})$.
Therefore, $l_{[z:w]} \in X$ iff there exists $i\neq j$ in $\lbrace 1,..,d\rbrace $  such that $f(l_{[z_i:w_i]}\cap U_\rho) = f(l_{[z_j:w_j]}\cap U_\rho)$.
 \def\svgwidth{10cm}
\begin{figure}[!h]
\begin{center}
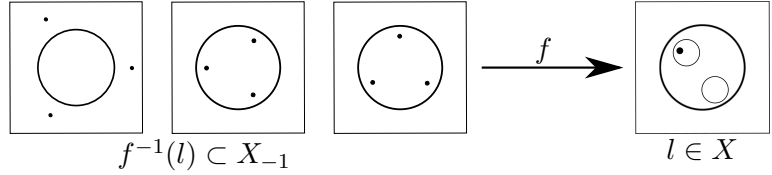
\caption{The problem that can occur when a line is in $X$.}
\end{center}
\end{figure}

A point in $f(U_\rho)$ is of the form $[P(z,w):Q(z,w):t^d + \varepsilon R(z,w)]$ with $|t|<\rho \max(|z|,|w|)$.
A sufficient condition for a line $l_{[z:w]}$ not to be in $B_{-1}$ is that there exists $i,j\in \lbrace 1,..,d \rbrace$ and $k\in \lbrace 1,..,d-1 \rbrace$ such that  $[P(z_i,w_i):Q(z_i,w_i):\varepsilon R(z_i,w_i)] = [P(z_j,w_j):Q(z_j,w_j):\varepsilon \e^{\frac{2ik\pi}{d}} R(z_j,w_j)]$. Possibly $i=j$ in which case this correspond to $R(z_i,w_i) = 0$ and necessarily, as in $l_{[z:w]}\cap U_\rho$, there will be points in $f(l_{[z:w]})$ that have $d$ preimages in $l_{[z:w]}\cap f(U_\rho)$.
Thus, two problems  can occur (when $i=j$ or not).
Denote by $Y_{-2}$ the set of lines $l_{[z_i:w_i]}, l_{[z_j:w_j]}$ which satisfy this condition and let $Y_{-1}=f(Y_{-2})$ and $Y=f_\infty^2(Y_{-2})$.

By definition of $X$ and $Y$, 
\begin{eqnarray} 
\label{def X}l_{[z:w]}\in X &\Longleftrightarrow &\exists  i\neq j \in \lbrace 1,..,d \rbrace \text{ such that } \\
& &R(z_i,w_i) - R(z_j,w_j)=0 \Longleftrightarrow l_{[z_i:w_i]}, l_{[z_j:w_j]} \in X_{-1} \nonumber\\
\label{def Y} f(l_{[z:w]})\in Y &\Longleftrightarrow  &\exists  i,j\in \lbrace 1,..,d \rbrace, \, k\in \lbrace 1,..,d-1 \rbrace  \text{ such that }\\
& & R(z_i,w_j) - \e^{\frac{2ik\pi}{d}} R(z_j,w_j)=0 \Longleftrightarrow l_{[z_i:w_i]}, l_{[z_j:w_j]} \in Y_{-2}.\nonumber
\end{eqnarray}
 \def\svgwidth{6cm}
\begin{figure}[!h]
\begin{center}
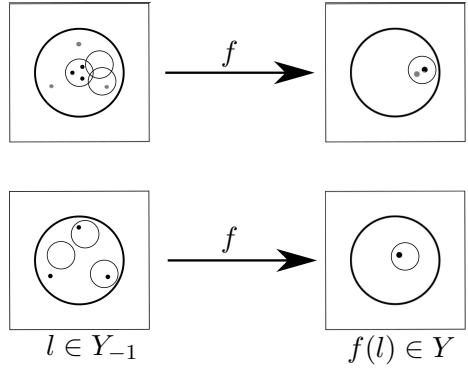
\caption{The two problems that can occur when a line is in $Y$.
 The black (resp. grey) points in $l$ are the preimages of the  black (resp. grey) point in $f(l)$.}
\end{center}
\end{figure}

By definition, we have that $X \cup Y \subset (A \cap B)^c$.
In fact, $l\in X$ (resp. $l\in Y$) precisely when $l$ is not in $A$ (resp. $B$) for any $\varepsilon$.

Denote by $f^+_\infty$ the map $f^+_\infty([z:w])=[P(z,w): Q(z,w):R(z,w)]$ and by $pr_1$ the map $pr_1:(l,l') \mapsto l$.
As $(l,l')\in X_{-1}$ (resp. $Y_{-2}$) iff $(l',l)\in X_{-1}$ (resp. $Y_{-2}$), we will, sometimes, identify $X_{-1}, Y_{-2}$ with $pr_1(X_{-1}), pr_1(Y_{-2})$, particularly in the examples. 

We note for further reference that, by identifying $l_{[z:w]}$ to $[z:w]\in \proj^1$, $X_{-1}$ and $Y_{-2}$ may be written formally as :
\begin{equation}\label{X-1 Y-2} 
\begin{cases}
X_{-1}=& \left\lbrace (p,q)\in \proj^1\times \proj^1 ; p\neq q, f_\infty^+(p)=f_\infty^+(q)\right\rbrace \\   & \bigcup \left\lbrace (p,p)\in \proj^1\times \proj^1 ; p \text{ is a critical point of } f_\infty\right\rbrace\\
\\
Y_{-2}= & \left\lbrace (p,q)\in \proj^1\times \proj^1 ; f_\infty^+(p)\neq f_\infty^+(q), f\circ f_\infty^+(p)=f\circ f_\infty^+(q), f_\infty(p)=f_\infty(q)\right\rbrace 
\\ 
& \bigcup \left\lbrace (p,p)\in \proj^1\times \proj^1 ; R(p)=0\right\rbrace\\
\\
X = &f_\infty(pr_1(X_{-1}))\\
Y =  &f^2_\infty(pr_1(Y_{-2})).
\end{cases}
\end{equation}

The next proposition is a sufficient condition to be of small topological degree on the attracting set $\A$.
Recall that $A,B,X,Y$ are sets of lines passing through $[0:0:1]$, that we identify such a line with a point in $L_\infty \simeq \proj^1$ and that the dynamics of the pencil of these lines  is the same as that of $f_\infty:[z:w]\rightarrow [P(z,w):Q(z,w)]$. 
Note that $X_{-1}, Y_{-2},X,Y$ only depend on the choice of $(P,Q,R)\in \F_d$, not on $\varepsilon$.

\begin{prop}\label{prop petit deg top}

Denote by $Z$ and $\mathscr{Z}$ the sets $Z=X \cup Y$ and $\mathscr{Z}=\left( f_\infty ^{-1}(Z)  \right) \bigcap Z$.
If $(P,Q,R)\in \F_d$ are such that

\begin{enumerate}
\item \label{cond ens. distincts}$X $ and $ Y$ are disjoint
\item \label{cond pt qui retombe} $ \left( f_\infty ^{-1} (\mathscr{Z})  \right) \bigcap Z = \emptyset$, or equivalently  $  f_\infty ^{-2} (Z) \bigcap f_\infty ^{-1} (Z)  \bigcap Z = \emptyset$
\end{enumerate}
then for  small enough $\varepsilon \neq 0$ there exists $\rho$ such that $f^3$ is of small topological degree on $f(U_\rho)$, with $f$ define by \eqref{def f}.
\end{prop}

\begin{rmq}
By \eqref{X-1 Y-2}, it is clear that $X_{-1},Y_{-2},X,Y$ are non empty algebraic set. Moreover, if the condition \ref{cond ens. distincts} of Proposition \ref{prop petit deg top} is satisfied, then $X_{-1},Y_{-2},X,Y$ are finite.
\end{rmq}

See Section \ref{section version simple} for a simplified version of these conditions.

\begin{proof}[Proof of Proposition \ref{prop petit deg top}]
For $\varepsilon$ as before, we fix $\rho = 4\frac{\varepsilon}{\alpha}$, as in \eqref{choix rho}, so that $f(U_{\rho}) \Subset U_{\rho}$.
Denote by $X^r$ (resp. $Y^r,Z^r,\mathscr{Z}^r$) the $r$-neighbourhood of $X$ (resp. $Y,Z,\mathscr{Z}$) in $L_\infty$.
We choose $r$ small enough such that $X^r$ (resp. $Y^r,Z^r,\mathscr{Z}^r$) is the union of disjoint open discs centered at the points of $X$ (resp. $Y,Z,\mathscr{Z}$) and
such that we have that :
\begin{itemize}
\item $X^r \cap Y^r =\emptyset$, by hypothesis \eqref{cond ens. distincts},
\item $f_\infty^{-1}(Z^r) \cap Z^r \subset \mathscr{Z}^r $, 
since $\left( f_\infty ^{-1}(Z)  \right) \bigcap Z= \mathscr{Z}$,
\item and $f_\infty^{-1}(\mathscr{Z}^r ) \cap Z^r=\emptyset$, by hypothesis \eqref{cond pt qui retombe}.
\end{itemize}
\begin{lemme}
Fix $\rho$ and $r$ as above. Then if $\varepsilon\neq 0$ is small enough, we have that 
\begin{equation}
A^c\subset X^r, \, B^c\subset Y^r, \\
\text{ and } X\cup Y=Z\subset (A\cap B)^c\subset Z^r.
\end{equation}
\end{lemme}

\begin{proof}
Indeed, the first inclusion is due to the definition of $X$ and $Y$.
Regarding the second inclusion, recall that $l_{[z_1:w_1]},.. ,l_{[z_d:w_d]}$ denoted the preimages (with possible repetitions) of a line $l_{[z:w]}$.

By \eqref{def X} and \eqref{def Y}, since $\proj^1 \setminus \left(  f^{-1}(X^r) \cup f^{-2}(Y^r)  \right)$ is compact and 
$X_{-1} \subset f_\infty^{-1}(X^r)$ and $Y_{-2} \subset f_\infty^{-2}(Y^r)$,
there exists $\gamma>0$ such that if $l_{[z_i:w_i]}, l_{[z_j:w_j]} \notin f^{-1}(X^r)$ with $i\neq j$  then 
$$| R(z_i,w_i) -   R(z_j,w_j) |>  \gamma (\max (|z_i|,|w_i|)+\max (|z_j|,|w_j|))^d$$
 and if $l_{[z_i:w_i]}, l_{[z_j:w_j]} \notin  f^{-2}(Y^r)$, ($i,j$ can be equal here) then for each $k\in \lbrace 1,..,d-1 \rbrace$ 
 $$| R(z_i,w_i) - \e^{\frac{2ik\pi}{d}}  R(z_j,w_j) |> \gamma (\max (|z_i|,|w_i|)+\max (|z_j|,|w_j|))^d.$$

Since $\rho = 4\frac{\varepsilon}{\alpha}$, $d\geq 2$ and $\gamma$ is independent of $d$ and $\varepsilon$, if $\varepsilon$ is small enough 
then $\varepsilon \gamma > \rho^d$, which implies that
if $l_{[z:w]}\notin X^r$ and $i,j$ are distinct then
\begin{equation}\label{ineg V0}
|\varepsilon R(z_i,w_i) -  \varepsilon R(z_j,w_j) |>  \rho^d \left(\max (|z_i|,|w_i|)+\max (|z_j|,|w_j|) \right)^d
\end{equation}
and if $l_{[z:w]}\notin f_\infty^{-1}(Y^r)$, ($i,j$ can be equal here) then for all $k\in \lbrace 1,..,d-1 \rbrace$
\begin{equation}\label{ineg V1}
 |\varepsilon R(z_i,w_i) - \e^{\frac{2ik\pi}{d}}  \varepsilon R(z_j,w_j) |>  \rho^d (\max (|z_i|,|w_i|)+\max (|z_j|,|w_j|))^d.
\end{equation}

If $l_{[z:w]}\notin X^r$ then for all $i,j$, $l_{[z_i:w_i]}, l_{[z_j:w_j]} \notin  f_\infty^{-1}(X^r)$ with $i\neq j$, thus,
 from \eqref{cond A0} and \eqref{ineg V0}, $l_{[z:w]}\in A$.
If $l_{[z:w]}\notin f_\infty^{-1}(Y^r)$ then for all $i,j$, we have that $l_{[z_i:w_i]}, l_{[z_j:w_j]} \notin  f_\infty^{-2}(Y^r)$,
thus, from \eqref{cond A1} and \eqref{ineg V1}, $l_{[z:w]}\in B_{-1}$.
Consequently, if $l_{[z:w]}\notin Y^r$, no preimage of $l_{[z:w]}$ is in $f_\infty^{-1}(Y^r)$, thus
all preimages of $l_{[z:w]}$ are in $B_{-1}$, i.e. $l_{[z:w]}\in B$.
Therefore, if $l_{[z:w]}\notin Z^r$ then $l_{[z:w]}\in A\cap B$, i.e.
 $(A \cap B)^c \subset Z^r$.
 
 This finish the proof of the lemma.
 \end{proof}

To finish the proof of Proposition \ref{prop petit deg top} we will prove that a point $p\in f^4(U_\rho)$ has at most $d^2$ preimages under $f^3$ (which lie in $ f(U_\rho)$). Denote by $\pi:\proj^2\setminus \lbrace [0:0:1]\rbrace\rightarrow L_\infty $ the projection defined by $\pi([z:w:t])=[z:w]$.

The proof is summarized the following diagram :
\\
\scalebox{0.8}{
\xymatrix{
 & \card(f^{-1}(p)\cap f(U_\rho))=1 \ar@{-}[rr]& & \card(f^{-2}(p)\cap f(U_\rho))\leq d\ar@{-}[r] & \card(f^{-3}(p)\cap f(U_\rho))\leq d^2 \\
1 \ar@{-}[ru]+L^{\Case 1.}\ar@{-}[rd]+L^{\Case 2.} & &&\card(f^{-2}(p)\cap f(U_\rho))\leq d\ar@{-}[r] & \card(f^{-3}(p)\cap f(U_\rho))\leq d^2 \\
 &\card(f^{-1}(p)\cap f(U_\rho))\leq d\ar@{-}+R;[rru]+L^{\Case 2.1.}\ar@{-}[rr]^{\Case 2.2.} & &  \card(f^{-2}(p)\cap f(U_\rho))\leq d^2 \ar@{-}[r] & \card(f^{-3}(p)\cap f(U_\rho))\leq d^2
}
}
\\

A point in $\pi^{-1}( (Z^r)^c ) \cap f^2 (U_\rho)$ has one preimage in $f (U_\rho)$.
Moreover, the complementary of $B$ is contained in $Y^r$ and the complement of $A$ is contained in $X^r$.
Since $X^r \cap Y^r =\emptyset$, 
we infer that a point in $ \pi^{-1}(Z^r) \cap f^2 (U_\rho)$ has at most $d$ preimages in $f (U_\rho)$.
Thus, a point has at most $d$ preimages under $f$ in $f (U_\rho)$.
\\

For the sake of simplicity, until the end of the proof, without precision by preimage we mean preimage in $f (U_\rho)$.
Pick a point $p\in f^4(U_\rho)$.
\\
Case 1. If $p\notin \pi^{-1}(Z^r)\cap f(U_\rho)  $ then $p\in \pi^{-1}(A\cap B)$, so $p$ has one preimage under $f$ and $\card \lbrace  f^{-3}(p) \cap U_\rho  \rbrace \leq d^2$.
\\
Case 2.  If $p\in \pi^{-1}(Z^r)\cap f(U_\rho)$ then $p$ has at most $d$ preimages under $f$.
For each preimage of $p$ we have two cases :
\\
Case 2.1. Either it is still in $\pi^{-1}(Z^r)\cap f(U_\rho)$. Thus it has $d$ preimages under $f$ which are in 
$ \pi^{-1}(\mathscr{Z}^r) $ (because $f_\infty^{-1}(Z^r) \cap Z^r \subset \mathscr{Z}^r $),
 but since $f_\infty^{-1}(\mathscr{Z}^r ) \cap 
Z^r=\emptyset$, these preimages leave $\pi^{-1}(Z^r)\cap f(U_\rho)$ and thus have only one preimage under $f$.
It follows that this preimage of $p$ has at most $d$ preimages under $f^2$.
\\
Case 2.2. Or it is outside $\pi^{-1}(Z^r)\cap f(U_\rho)$, then it has one preimage. It follows that this preimage of $p$ has  at most $d$ preimages under $f^2$.

Therefore, $\card \lbrace  f^{-3}(p)\cap f(U_\rho)   \rbrace \leq d^2$ in all cases.
Consequently, $f^3$ is of small topological degree on $f(U_\rho)$.
This concludes the proof of the Proposition.
\end{proof}

Denote by $\Omega$ the subset  of $(P,Q,R)\in \F_d$ satisfying the hypotheses of Proposition \ref{prop petit deg top}. 
\\

At this point to finish the proof of Theorem \ref{theoreme} it only remains to prove that $\Omega$  is a Zariski open set and that it is not empty.
\\

\textbf{Step 3.} The subset $\Omega$ of $ \F_d$  is a Zariski open set.\\

This is simply an exercise in complex geometry.
Recall from \eqref{X-1 Y-2} that
$X=pr_1\circ (f_\infty\otimes f_\infty) (X_{-1})$, $Y=pr_1\circ (f_\infty^2\otimes f_\infty^2) (Y_{-2})$,
 $Z=X \cup Y$ and $\mathscr{Z}=\left( f_\infty ^{-1}(Z)  \right) \bigcap Z$, where $pr_1(p,q)=p$.
As already observed, $X,Y,Z$ only depend on the choice of $(P,Q,R)\in \F_d$ and not on the choice of $\varepsilon\neq 0$.
Recall that the conditions of Proposition \ref{prop petit deg top} are
\begin{enumerate}
\item $X$ and $ Y$ are disjoint 
\item $ \left( f_\infty ^{-1} (\mathscr{Z})  \right) \bigcap Z = \emptyset$  or equivalently $f_\infty ^{-2} (Z) \bigcap f_\infty ^{-1} (Z) \bigcap Z = \emptyset$.
\end{enumerate}
Set
$$
\begin{array}{ll}
\widetilde{X}_{-1} = & \lbrace (p,q,(P,Q,R))\in \proj^1\times \proj^1 \times \F_d; \, p\neq q, f_\infty^+(p)=f_\infty^+ (q)\rbrace \\
  & \bigcup 
\lbrace (p,p,(P,Q,R))\in \proj^1\times \proj^1 \times \F_d ;\, p \text{ is a critical point of } f_\infty\rbrace , \\
\\
 \widetilde{Y}_{-2}  = & \lbrace (p,q,(P,Q,R))\in \proj^1\times \proj^1 \times \F_d ;\, f_\infty^+(p)\neq f_\infty^+(q), f\circ f_\infty^{+}(p)=f\circ f_\infty^{+}(q), f_\infty(p)=f_\infty(q)\rbrace \\
 & \bigcup \lbrace (p,p,(P,Q,R))\in \proj^1\times \proj^1 \times \F_d; \, R(p)=0 \rbrace ,
\end{array}
$$
and denote by $\Phi$ the map $\Phi : (p,q,(P,Q,R))\mapsto (f_\infty (p),f_\infty (q), (P,Q,R))$ and by $\widetilde{X},\widetilde{Y},\widetilde{Z}$ the sets  $\widetilde{X}=\Phi (\widetilde{X}_{-1})$, $ \widetilde{Y}= \Phi^2(\widetilde{Y}_{-2})$ and $ \widetilde{Z}= \widetilde{X}\bigcup \widetilde{Y}$.
Thus $\left\lbrace (p,q,f_\infty (p),f_\infty (q), (P,Q,R));(p,q, (P,Q,R))\in \widetilde{X}_{-1} \right\rbrace$ is an algebraic subvariety of $\proj^1\times \proj^1 \times \proj^1\times \proj^1 \times \F_d$ as well as $\left\lbrace (p,q,f^2_\infty (p),f^2_\infty (q), (P,Q,R));(p,q, (P,Q,R))\in \widetilde{Y}_{-2} \right\rbrace$.
Therefore,
 $ \widetilde{X}, \widetilde{Y}$ and $ \widetilde{Z}$ are algebraic subvarieties in $\proj^1\times \proj^1 \times \F_d$, and $ \widetilde{X} \cap \widetilde{Y}$ and $\Phi^{-2}( \widetilde{Z}) \cap \Phi^{-1}( \widetilde{Z}) \cap  \widetilde{Z}$ are algebraic subvarieties too.
The complementary of $\Omega$ is in the image of $ \left( \widetilde{X} \cap \widetilde{Y} \right) \bigcup \left(\Phi^{-2}( \widetilde{Z}) \cap \Phi^{-1}( \widetilde{Z}) \cap  \widetilde{Z} \right)$ by the projection  $(p,q,(P,Q,R))\mapsto (P,Q,R)$, thus it is a Zariski closed  subset in $\F_d$.
\\

\textbf{Step 4.} The Zariski open set $\Omega$ defined in step 3 is not empty.
\\

Let us show that $f=[w^d+a z^d:b w^d+z^d:t^d + \varepsilon (z^d+z^{d-1}w)]$ satisfies the hypotheses of Proposition \ref{prop petit deg top} for almost every $a,b\in \C$ small enough. We first assume $|a|,|b|>0$.

If $l_{[z:w]}$ is a preimage of a line, the other preimages are $l_{[z:\e^{\frac{2ik\pi}{d}}w]}$ with $k\in \lbrace 1,..,d-1 \rbrace$. For $l\in \lbrace 1,..,d-1 \rbrace , k\in \lbrace 0,..,d-1 \rbrace$ put $$z_{k,l}=[\e^{\frac{2i(k+l)\pi}{d}}-1:1-\e^{\frac{2il\pi}{d}}].$$
We let the reader check that $X_{-1}=\lbrace  [0:1],[1:0]  \rbrace$ and 
\begin{eqnarray}\nonumber
Y_{-2} & =\lbrace  [\e^{\frac{2i(k+l)\pi}{d}}-1:1-\e^{\frac{2il\pi}{d}}] ;  l\in \lbrace 1,..,d-1 \rbrace , k\in \lbrace 0,..,d-1 \rbrace \rbrace\\ \nonumber
 & =  \lbrace  [0:1]\rbrace \bigcup \lbrace z_{k,l} ;  l\in \lbrace 1,..,d-1 \rbrace, k\in \lbrace 0,..,d-1 \rbrace, k+l\neq d \rbrace .
\end{eqnarray}
Thus
 \begin{eqnarray*}
&X & = \lbrace [1:b],[a:1]  \rbrace ,\\
&Y &=f_\infty^2(Y_{-2})
 =\lbrace  f_\infty([1:b])\rbrace \bigcup \lbrace f_\infty^2(z_{k,l}) ;  l\in \lbrace 1,..,d-1 \rbrace, k\in \lbrace 0,..,d-1 \rbrace, k+l\neq d \rbrace.
 \end{eqnarray*}

Suppose that $a,b\in \C^*$ are small and $a\sim b$.

If $k+l=d$ then $z_{k,l}=0=[0:1]$, $f_\infty(z_{k,l})=[1:b]$ and $f_\infty^2(z_{k,l})=[b^d+a:b^{d+1}+1]\approx [a:1]$,
otherwise $f_\infty(z_{k,l})\approx z_{k,l}^{-d}$ and $f_\infty^2(z_{k,l})\approx z_{k,l}^{2d}$. 

If $k+l\neq d$ then $z_{k,l}$ $f_\infty(z_{k,l})$ and
 $f_\infty^2(z_{k,l})$ are far from $[0:1],[1:0]$.
In particular, $\lbrace f_\infty(z_{k,l}) , f_\infty^2(z_{k,l})\rbrace \cap \lbrace  [0:1],[1:0],[1:b],[a:1],[b^d+a:b^{d+1}+1],[1+a^{d+1}:b+a^d]  \rbrace =\emptyset$.
We conclude that, if $a,b\in \C^*$ are small and $a\sim b$ then $X$ and $ Y$ are disjoint, i.e. $f$ satisfies the hypothesis \eqref{cond ens. distincts} of Proposition \ref{prop petit deg top}. To check assumption \eqref{cond pt qui retombe}, write
$$Z=X \cup Y=
\lbrace  [1:b],[a:1],f_\infty([1:b])\rbrace \bigcup \lbrace f_\infty^2(z_{k,l}); k+l\neq d \rbrace,$$
so that
$$f_\infty^{-1}(Z)=\lbrace  [0:1],[1:0],[1:\e^{\frac{2in\pi}{d}} b], \e^{\frac{2in\pi}{d}}f_\infty(z_{k,l}) | n\in \N  \rbrace.$$

The $z_{k,l}$ are independent of $a,b$, so we can choose $a,b$ so that
 $\lbrace \e^{\frac{2in\pi}{d}}f_\infty(z_{k,l}) \rbrace \bigcap \lbrace f_\infty^2(z_{k,l}) \rbrace = \emptyset$,
thus $f_\infty^{-1}(Z)\bigcap Z=\lbrace  [1:b] \rbrace$.
Since $f_\infty^{-1}([1:b])=\lbrace [0:1]\rbrace \notin Z$, we conclude  that $f$ satisfies the second assumption  of Proposition \ref{prop petit deg top}.

This completes the proof of Theorem \ref{theoreme}.
\end{proof}


We now prove Proposition \ref{prop exist attracteur non pp}.

\begin{proof}[Proof of Proposition \ref{prop exist attracteur non pp}]
We first show that $f=[(z-2w)^2:z^2:t^2 + \varepsilon (z^2+zw)]$ satisfies the hypotheses of Proposition \ref{prop petit deg top}.
The preimages of $L_{[\alpha:\beta]}$ are $L_{[2\sqrt{\beta}: \sqrt{\beta}-\sqrt{\alpha}]}$ and $L_{[2\sqrt{\beta}: \sqrt{\beta}+\sqrt{\alpha}]}$.
We let the reader check that $X=\lbrace [0:1],[1:0] \rbrace$ and $Y=\lbrace f_\infty ([1:-1]),f_\infty ([1:9])\rbrace=\lbrace [9:1],[17^2:1] \rbrace$, which are disjoint.

Thus $Z= \lbrace [0:1],[1:0],[9:1],[17^2:1] \rbrace$ and $ f_\infty ^{-1}(Z) =\lbrace [2:1],[0:1],[1:2],[1:-1],[1:9],[1:8] \rbrace$, therefore
$\mathscr{Z}=\left( f_\infty ^{-1}(Z)  \right) \bigcap Z=\lbrace [0:1],[1:9] \rbrace$.
Finally, $f_\infty^{-1}(\mathscr{Z})=\lbrace [2:1],[1:2],[1:-1] \rbrace $ and $f$ satisfies the hypotheses of Proposition \ref{prop petit deg top}.

Thus, by Proposition \ref{th potentiel bourne} and \ref{prop petit deg top}, $f$ has an attracting set $\A\subset \proj^2$ supporting a positive closed current of bidegree (1,1) which admits a bounded quasi potential.

From the work of \cite{FSdyn} Lemma 2.12 and \cite{R} corollary 4.15,  the natural extension $\hat{f}_\infty$ of $f_\infty$ is semi-conjugate to $f$ restricted to $\A$. 
Since $J_\infty=L_\infty$, $f_\infty$ is topologically mixing on $L_\infty$, thus $f$ is topologically mixing on $\A$ and $\A$ is an attractor.
\end{proof}

\section{Non algebraic attracting sets in higher codimension} \label{section pluri non alg}
In this part, we will explain why the condition of being of small topological degree is not enough to ensure non pluripolarity in higher codimension. 
For this we exhibit attracting sets of small topological degree in $\proj^3$ (of codimension 2) which are contained in a hyperplane (see Theorem \ref{th att cont in hyperplan}).
 We also construct the first explicit example of a Zariski dense attracting set of higher codimension (see Theorem \ref{th ex Z dense}).

\begin{theoreme}\label{th att cont in hyperplan}
For a generic choice of $(P,Q,R)\in \F_d$, there exists $\varepsilon_1(P,Q,R), \varepsilon_2(P,Q,R)>0$ such that for all $0<|\varepsilon_1|<\varepsilon_1(P,Q,R)$ and $0<|\varepsilon_2|<\varepsilon_2(P,Q,R)$, the mapping 

$$\begin{array}{ccccl}
f & : & \proj^3 & \longrightarrow &  \proj^3  \\
& & [z:w:t:u] & \mapsto &  [P(z,w) :Q(z,w):t^d+\varepsilon_1 R(z,w):u^d+\varepsilon_2 Q(z,w) -\varepsilon_2^d w^d] \\
\end{array}$$
is of small topological degree on an attracting set 
$\A$ of dimension $1$ which is contained in $\lbrace [z:w:t:u];u=\varepsilon_2 w \rbrace\simeq \proj^2$.
 
 In particular, $\A$ is non algebraic nevertheless it is contained in a hyperplane.
\end{theoreme}

\begin{rmq}
If $f: [z:w:t]  \mapsto   [P_0(z,w,t) :P_1(z,w,t):P_2(z,w,t)]$ admits an attracting set $\A$ of dimension 1 then $\tilde{\A} =\lbrace [z:w:t:u] \, |\, [z:w:t]\in \A,u=0 \rbrace$ is an attracting set for $\tilde{f} : [z:w:t:u] \mapsto  [P_0(z,w,t) :P_1(z,w,t):P_2(z,w,t):u^d]$. But $\tilde{f}$ is not of small topological degree on $\tilde{\A}$ even if $f$ is of small topological degree on $\A$. 
This is why we need a non-trivial construction in Theorem \ref{th att cont in hyperplan}.
\end{rmq}

\begin{proof}
Denote by $f_1,f_2$ the maps
$$
\begin{array}{ll}
f_1:&[z:w:t]\mapsto[P(z,w) :Q(z,w):t^d+\varepsilon_1 R(z,w)],\\
f_2:&[z:w:u]\mapsto[P(z,w) :Q(z,w) : u^d+\varepsilon_2 Q(z,w) -\varepsilon_2^d w^d]
\end{array}
$$
and denote by $U_1,U_2$ the open sets
 $$U_1=\lbrace [z:w:t];|t|<\rho_1\max(|z|,|w|) \rbrace , U_2=\lbrace [z:w:u];|u-\varepsilon_2 w|<\rho_2\max(|z|,|w|) \rbrace .$$
First we choose $(P,Q,R)$ in the non empty Zariski open set of $\F_d$ such that the map $f_{1}$  satisfies the hypotheses of Proposition \ref{prop petit deg top} and fix $\varepsilon_1(P,Q,R)>|\varepsilon_1|>0$ small enough and $\rho_1>0$ such that $f_{1}(U_{1})\Subset U_{1}$, see Step 1 of the proof of Theorem \ref{theoreme}.
By modifying the end of Step 2 of the proof of Theorem \ref{theoreme}, a point in $f_{1}^7 (U_{1})$ has at most $d^4$ preimages in $f_{1}(U_{1})$ under $f^6_1$.

It is clear that the hyperplane $\lbrace [z:w:t:u];u=\varepsilon_2 w \rbrace$ is invariant under $f$ and we see that the dynamics on it is the same as that of $f_{1}$. Denote by $\rho_2$ the constant $\rho_2=c\varepsilon_2$, with $c>0$, and by $U$ the set
$$U=\left\lbrace [z:w:t:u]; |t|<\rho_1 \max(|z|,|w|), |u-\varepsilon_2 w|<\rho_2 \max(|z|,|w|)  \right\rbrace.$$
We choose $\varepsilon_2(P,Q,R)$ small enough such that $((1+c)^d+1)\varepsilon_2(P,Q,R)^d<c\varepsilon_2(P,Q,R)$ and we will fix $c$ later.
 We let the reader check that $f(U)\Subset U$ and denote by $\A=\bigcap f^n(U)$, hence $\A$ is an attracting set of dimension 1.

We now prove that, for a generic choice of $(P,Q,R)\in \F_d$, $f^6$ is of small topological degree on $f(U)$.

Let us further assume that for all $i\in \lbrace 1,..,6 \rbrace$ $f_\infty^i([1:0])\neq [1:0]$ or equivalently  
$[1:0]\notin f_\infty^{-i}([1:0])$ where $f_\infty:\proj^1\rightarrow \proj^1$ is defined by $f_\infty([z,w])=[P(z,w):Q(z,w)]$.
 Arguing as in Step 3 of the proof of Theorem \ref{theoreme}, we see that this condition is algebraic. To show that it is generically satisfied, we need to find an example. 
Let us give an example.
We know by Step 4 of the proof of Theorem \ref{theoreme} that for almost every $(a,b)\in \C^2$ with $|a|,|b|>0$ small enough $f_{1}:[z:w:t]\mapsto [w^d+a z^d:b w^d+z^d:t^d + \varepsilon_1 (z^d+z^{d-1}w)]$ satisfies the hypotheses of Proposition \ref{prop petit deg top}. For all $(a,b)\in \R^2$ such that $a,b>0$, we have that 
for each $i\in \N^*$ $f^i_{1} ([1:0])\neq [1:0]$, because if $(z,w)\neq (0,0)$ are non negative real numbers then $w^d+a z^d,b w^d+z^d$ are positive.
Therefore, it is a generic condition.

Fix $\D$ a disc centered at $[1:0]$ in $\proj^1$ such that for all $i\in \lbrace 1,..,6 \rbrace$ $f_\infty^{-i}(\D)\cap \D = \emptyset$, fix $c>0$ such that if $[z:w]\notin \D$ then for each $l\in \lbrace 1,..,d-1 \rbrace$ 
\begin{equation}
|\e^{\frac{2il\pi}{d}}w-w|>3c\max(|z|,|w|).     \nonumber
\end{equation}
This implies that if $[z:w:u]\in U_{2}$ and $[z:w]\notin \D$ then for each $l\in \lbrace 1,..,d-1 \rbrace$,
\begin{equation}
|\e^{\frac{2il\pi}{d}}u-\varepsilon_2 w|>\rho_2\max(|z|,|w|)  
 \text{, i.e. } [z:w:\e^{\frac{2il\pi}{d}} u]\notin U_{2}.    \nonumber
\end{equation}

Let $[z:w:u]\in f^2_2(U_{2})$ and $[z_{-1}:w_{-1}]$ such that $f_\infty([z_{-1}:w_{-1})=[z:w]$.
If $[z_{-1}:w_{-1}]\notin \D$ there exists a unique $u_{-1}$, such that $f_2([z_{-1}:w_{-1}:u_{-1}])=[z:w:u]$ and $[z_{-1}:w_{-1}:u_{-1}]\in f_2(U_{2})$.
 Otherwise, if $[z_{-1}:w_{-1}]\in \D$ there exists at most $d$ such values $u_{-1}$.
Thus, if we fix $[z_{-6}:w_{-6}]$  such that $f_{\infty}^6([z_{-6}:w_{-6}])=[z:w]$, since for all $i\in \lbrace 1,..,6 \rbrace$ $f_\infty^{-i}(\D)\cap \D = \emptyset$, there exists at most one $l\in \lbrace 0,..,6 \rbrace$ such that $f_\infty^l([z_{-6}:w_{-6}])\in \D$.
So there  exists at most $d$ values $u_{-6}$, such that  $f^6_2([z_{-6}:w_{-6}:u_{-6}])=[z:w:u]$ and $[z_{-6}:w_{-6}:u_{-6}]\in f_2 (U_{2})$.

Let $[z:w:t:u]\in f^7(U)$, then $ [z:w:t]\in f_{1}^7(U_{1})$, so there exists at most $d^4$ preimages $[z_{-6}:w_{-6}:t_{-6}]$ in $f_{1}(U_{1})$  such that $f_{1}^6([z_{-6}:w_{-6}:t_{-6}])=[z:w:t]$.
For each choice $[z_{-6}:w_{-6}]$, there  exists at most $d$ values $u_{-6}$, such that 
$f^6_2([z_{-6}:w_{-6}:u_{-6}])=[z:w:u]$ and $[z_{-6}:w_{-6}:u_{-6}]\in f_2 (U_{2})$.
Then for each choice $[z_{-6}:w_{-6}:t_{-6}]\in f_1(U_1)$, there  exists at most $d$ values $u_{-6}$, such that 
$f^6([z_{-6}:w_{-6}:t_{-6}:u_{-6}])=[z:w:t:u]$ and $[z_{-6}:w_{-6}:t_{-6}:u_{-6}]\in f (U)$.
In conclusion, a point in $f^7(U)$ has at most $d^5$ preimages under $f^6$ in $f(U)$.
Thus $f^6$ is of small topological degree on $f(U)$.
\end{proof}

We now turn to the Zariski dense example of codimension 2 in $\proj^3$.
We first recall some results about H\'enon-like maps.

Let $W$ be an open set and $\D$ an open disc, denote by $\B$ and $\partial_h \B $ the sets $ \B= W\times \D$ and
 $\partial_h \B = W \times \partial \D$, $\partial_v \B = (\partial W)\times \D$, $N(\B)$ a neighbourhood of $\B$ and $d_t$ the maximum number of preimages in $\B$ of a point in $f(\B)\cap\B$.

\begin{defi}
The map $f:N(\B) \rightarrow \C^2$ is called a horizontal-like mapping if
  \begin{enumerate}
  \item $f(\partial_v \B)\cap \B=\emptyset$
  \item $ f(\overline{\B})\cap \partial \B \subset \partial_v \B$,
  \end{enumerate}
If, furthermore, $f$ is injective, $f$ is said to be Hénon-like.
\end{defi}

See \cite{Du} for some proprieties of H\'enon-like and horizontal-like mappings.
A current $T$ is said to be horizontal if its support is contained in $W\times\D_{1-\varepsilon}$.

\begin{prop}
Let $f$ be a horizontal-like map. Then there exists an integer $d \geq 1$ such that
for every normalized horizontal positive closed current $T$ in $\B$, $\frac{1}{d}f_*T$ is normalized, horizontal, positive and closed in $\B$.

The integer $d$ is called the degree of $f$
\end{prop}

\begin{proof}
See \cite{dds} or \cite{Du}.
\end{proof}

\begin{theoreme}\label{th pot holder}
Let $f$ be a Hénon-like map in $\B$. If $T$ is an invariant current by $\frac{1}{d}f_*$, then the potentials of $T$ are Hölder continuous.
\end{theoreme}

\begin{proof}
This is identical to Theorem 2.12 in \cite{Du}.
\end{proof}

\begin{defi}
Let $f$ be a horizontal-like map of degree $d$. If the number of preimages in $\B$ of a point in $f(\B)\cap \B$ is strictly less than $d$, $f$ is said to be of small topological degree.  
\end{defi}

\begin{lemme}\label{lemme henon}
For $\varepsilon\neq 0$ small enough, there exists $\rho$ such that the map
$$f[z:w:t]\mapsto[z^2+0.1 w^2:w^2:t^2+\varepsilon(z^2+zw)]$$ is a H\'enon-like map on $f(\B)$ where $\B=\lbrace [z:w:t];0.8|w|<|z|<1.2|w|,|t|<\rho \max(|z|,|w|) \rbrace $.
Furthermore, denoting by $\A=\bigcap f^n(U_\rho)$, a point in $\A \cap \B$ has a unique preimage in $\A$ and this preimage is still in $\A\cap \B$.
\end{lemme}

\begin{proof}
Denote by $U_\rho$ the set$U_\rho=\lbrace [z:w:t];|t|<\rho \max(|z|,|w|) \rbrace$, by the above, it follows that for $\varepsilon$ small enough, there exists $\rho$ such that  $f(U_\rho)\Subset U_\rho$.
Denote by $W$ the set $W=\lbrace [z:w];0.8|w|<|z|<1.2|w|\rbrace$, an easy computation shows that $f_\infty^{-1}(W)\Subset W\Subset f_\infty(W)$, where $f_\infty [z:w]\mapsto[z^2+0.1 w^2:w^2]$. 
Thus $f$ is a horizontal-like map on $\B$.
With notation as in Step 2 of Theorem \ref{theoreme}, we get that $X=\lbrace [0.1:1],[1:0] \rbrace$ and $Y=[1.31:1]$. Reducing $\varepsilon$ if necessary, we get that $W\subset A\cap B$ (see Step 2 of Theorem \ref{theoreme}), thus $f$ is injective on $f(\B)$.

As $f$ is H\'enon-like on $f(\B)$, if $[z:w:t]\in \A \cap \B$ then $[z:w:t]\in \A \cap f(\B)$ and $[z:w]\in W$ thus there exists a unique $[z_{-1}:w_{-1}:t_{-1}]\in \A$ such that $f([z_{-1}:w_{-1}:t_{-1}])=[z:w:t]$. Moreover, as $f_\infty^{-1}(W)\Subset W$, $[z_{-1}:w_{-1}] $ is still in $W$.
\end{proof}

\begin{theoreme}\label{th ex Z dense}
There exists $\varepsilon_1,\varepsilon_2\neq 0$ such that the map
$$f:[z:w:t_1:t_2]\mapsto[z^2+0.1 w^2:w^2:t_1^2+\varepsilon_1(z^2+zw):t_2^2+\varepsilon_2 (z^2+zw)]$$
admits a Zariski dense attracting set $\A$ close to the line $\lbrace [z:w:t_1:t_2]; t_1=0 ; t_2=0  \rbrace$.
\end{theoreme}

\begin{notation}
Denote by $pr_1,pr_2$ the projections defined by $pr_1([z:w:t_1:t_2])=[z:w:t_1]$ and $pr_2([z:w:t_1:t_2])=[z:w:t_2]$.
\end{notation}

\begin{proof}
For $i=1,2$ denote by $f_i$ the map $f_i:[z:w:t_i]\mapsto[z^2+0.1 w^2:w^2:t_i^2+\varepsilon_1(z^2+zw)]$, by $U_i$ and $V_i$ the sets $U_i=\lbrace [z:w:t_i];|t_i|<\rho_i\max(|z|,|w|) \rbrace$ and $V_i=\lbrace [z:w:t_i];0.8|w|<|z|<1.2|w|,|t_i|<\rho_i \max(|z|,|w|) \rbrace $.
We choose $\varepsilon_i$ and $\rho_i$ such that $f_i(U_i)\Subset U_i$. Denote by $\A_i=\bigcap f_i^n(U_i)$.
and by $\tau_i$ the positive closed current given by Theorem \ref{th dinh} for $f_i$. 
By Theorem \ref{th pot holder} and Lemma \ref{lemme henon}, there exists $\theta_i>0$ such that, in any neighbourhood $N$ of a point in $\A_i\cap V_i$, $\supp(\tau_i)\cap N$ has Hausdorff dimension at least $2+\theta_i$. 
On the other hand, J.E. Forn\ae ss and N. Sibony \cite[Prop. 2.17]{FSdyn} proved that if  $\varepsilon_2$ is small enough then $\A_2$ has Hausdorff dimension less than $2+\frac{\theta_1}{2}$.

Denote by $U$ the set $U=\lbrace [z:w:t_1:t_2];|t_1|<\rho_1\max(|z|,|w|),|t_2|<\rho_2\max(|z|,|w|) \rbrace$, thus $f(U)\Subset U$.
The attracting set $\A=\bigcap f^n(U)$ is of dimension 1, i.e. of codimension 2.
We have that $pr_1(\A)=\A_1$ and $pr_2(\A)=\A_2$, thus $\A$ is non algebraic. 
We will use the discrepancy between the Hausdorff dimension of $\A_1$ and $\A_2$ to show that $\A$ must be Zariski dense.

Denote by $(\hat{\proj}^1,\hat{f}_\infty)$ the natural extension of $(\proj^1, f_\infty) $ where $f_\infty:[z:w]\mapsto [P(z,w):Q(z,w)]$, see section 3 of \cite{R} for an introduction.
Let $p=[z:w]\in \proj^1$ and denote by $l_p$ the line (resp. hyperplane) passing through $[0:0:1]$ and $[z:w:0]$ (resp. $[0:0:1:0]$, $[0:0:0:1]$ and $[z:w:0:0]$). We have already seen that $l_p\cap U_i$ is a disk and its image is relatively compact in $l_{f_\infty (p)}\cap U_i$, see Step 1 of Theorem \ref{theoreme}.
Thus if $(p_{-n})\in \hat{\proj}^1$ is a sequence of preimages of $p$ then $\bigcap_{n\geq 0} f_i^n (l_{p_{-n}}\cap U_i)$ is a unique point in $l_p\cap \A_i$. 
Therefore, this defines a continuous and onto map ${\pi_i}: \hat{\proj}^1 \rightarrow \A_i$, see \cite[Prop. 4.13]{R} or \cite[Lemma 2.8]{FSdyn}.
We have that if $p_{-n}=[z_{-n}:w_{-n}]$ then 
$\pi_i((p_{-n}))=\underset{n\rightarrow \infty}{\lim} f_i^n([z_{-n}:w_{-n}:0])$.
 We can also define a continuous and onto map ${\pi}: \hat{\proj}^1 \rightarrow\A$ by replacing disks by balls in the previous argument. This map satisfies $pr_i\circ \pi=\pi_i $, for $i\in \lbrace 1;2\rbrace$. We have the three commutative diagrams :

$$
\xymatrix{
\hat{\proj}^1 \ar[r]^{\hat{f}_\infty} \ar[d]^{\pi_1} & \hat{\proj}^1\ar[d]^{\pi_1}\\
\A_1 \ar[r]^{f_1}  & \A_1 \\
}
\qquad
\qquad
\xymatrix{
\hat{\proj}^1 \ar[r]^{\hat{f}_\infty} \ar[d]^{\pi_2} & \hat{\proj}^1\ar[d]^{\pi_2}\\
\A_2 \ar[r]^{f_2} & \A_2 \\
}
\qquad
\qquad
\xymatrix{
\hat{\proj}^1 \ar[r]^{\hat{f}_\infty} \ar[d]^\pi & \hat{\proj}^1\ar[d]^\pi\\
\A \ar[r]^{f} & \A \\
}
$$

By Lemma \ref{lemme henon}, $f_i$ is H\'enon-like on $f_i(V_i)$ thus $\pi_i$ restricted to $\pi_i^{-1}(V_i\cap \A_i)$ and $\varphi_i = \pi\circ \pi_i^{-1}: \A_i\cap V_i \rightarrow \A\cap V$ are homeomorphisms, where 
$V=\lbrace [z:w:t_i];0.8|w|<|z|<1.2|w|,|t_1|<\rho_2 \max(|z|,|w|), |t_2|<\rho_2 \max(|z|,|w|) \rbrace $. Moreover, as $\pi_i((p_{-n}))=\underset{n\rightarrow \infty}{\lim} f_i^n([z_{-n}:w_{-n}:0])$ and $\pi((p_{-n}))=\underset{n\rightarrow \infty}{\lim} f^n([z_{-n}:w_{-n}:0:0])$, we have that $pr_i\circ \varphi_i=\id$.

%

Suppose that the Zariski closure of $\A$ is a hypersurface $M$.
As $pr_i(\A) = \A_i$ is non algebraic, $M$ cannot be the union of hyperplane passing through $[0:0:0:1]$

Let $W_1$ be an open set contained in $V_2$ such that 
$M\cap pr_1^{-1}(W_1)$ is the union of biholomorphic copies of $W_1$ included in the regular points of $M$.
 At least one of these copies intersect $\A$ 
in a set of Hausdorff dimension at least  $2+\theta_2 $, because $W_1\cap\A_1$ has Hausdorff dimension at least $2+\theta_2 $. 
Denote it by $D$ and
denote by $\Cr$ the irreducible component of $M$ which contains $D$.

If $pr_2(\Cr)$ is not a curve then, up to reducing $W_1$, $pr_{2|D}$ is a biholomorphism. 
Thus $\dim_H(pr_2(C\cap\A))\geq 2+\theta_2 $, but this contradicts  $pr_2(C\cap\A)\subset \A_2$ and 
$\dim_H(\A_2)\leq 2+\frac{\theta_1}{2}$.

Otherwise, let $p\in \A\cap \Cr$ and $W$ be an open set such that $p\in W\cap\Cr=W\cap M\subset D$. 
 As $\varphi_2$ is continuous there exists an open set $W_2\subset pr_2(W)$
such that $pr_2(p)\in W_2 $ and $\varphi_2(W_2\cap\A_2)\subset W$. Moreover, $\varphi_2(W_2\cap\A_2)\subset M$
 so $\varphi_2(W_2\cap\A_2)\subset \Cr$. Then $W_2\cap\A_2=pr_2(\varphi_2(W_2\cap\A_2))$ is included in the
  curve $pr_2(\Cr)$ but this contradicts the fact that $\dim_H (pr_2(W)\cap\A_2)\geq 2+\theta_2>2$.

We have reached a contradiction.
Therefore, $\A$ is Zariski dense in $\proj^3$.
\end{proof}
\begin{rmq}
For each $n\in \N$ the point $[1:0]$ has multiplicity $2^n$ for $f_\infty^n$ thus $f$ is not of small topological degree on $\A$.
\end{rmq}

\section{Further results and open problems}

\subsection{A simpler version of Proposition \ref{prop petit deg top}}\label{section version simple}
The original generic condition that we had in mind for small topological degree map on an attracting set (Proposition \ref{prop petit deg top})   was somewhat simpler. Unfortunately, we were only able to find corresponding examples in degree 2.

Indeed, Condition \ref{cond pt qui retombe} of Proposition \ref{prop petit deg top} can be replaced by the algebraic condition

\ref{cond pt qui retombe}'. $f^{-1}(Z)\cap Z = \emptyset$.
\\
 It is clear from the proof of Proposition \ref{prop petit deg top} that this implies that for all $|\varepsilon|>0$ small enough there exists $\rho$ such that $f^2$ is of small topological degree on $f(U_\rho)$.

\begin{ex} 
The map $f=[w^2+a z^2:b w^2+z^2:t^2 + \varepsilon (z^2+2zw+w^2)]$ satisfies the hypothesis \ref{cond ens. distincts} of the Proposition \ref{prop petit deg top} and \ref{cond pt qui retombe}' for almost every $(a,b)\in \C^2$.
In fact, we have that $X_{-1}=\lbrace  [0:1],[1:0]  \rbrace$ thus $X = \lbrace [1:b],[a:1]  \rbrace$ and $Y_{-2}=\lbrace  [i:1],[i:-1],[1:-1]   \rbrace$ thus
$$Y=f_\infty ^2 (Y_{-2})=\lbrace  [(1+b)^2+a(1+a)^2:b(b+1)^2+(1+a)^2],[(b-1)^2+a(1-a)^2:b(b-1)^2+(1-a)^2]  \rbrace.$$
Therefore, for almost every $(a,b)\in \C^2$, $X $ and $ Y$ are disjoint, i.e. $f$ satisfies the hypothesis \ref{cond ens. distincts} of \ref{prop petit deg top} and as $Z= X\cup Y$
$$Z=
\lbrace  [1:b],[a:1],[(1+b)^2+a(1+a)^2:b(b+1)^2+(1+a)^2],[(b-1)^2+a(1-a)^2:b(b-1)^2+(1-a)^2]  \rbrace$$
and
$$f_\infty ^{-1}(X \cup Y)=
\lbrace  [0:1],[1:0],[1+a:b+1],[1-a:b-1],[1+a:-b-1],[1-a:-b+1]  \rbrace,$$
for almost every $(a,b)\in \C^2$,
we have that $f^{-1}(Z)\cap Z=\left( f_\infty ^{-1}(X \cup Y)  \right) \bigcap \left( X \cup Y \right)=\emptyset$, thus  $f$ satisfies the condition \ref{cond pt qui retombe}'.
\end{ex}

\subsection{Generalisation of Theorem \ref{theoreme} to higher dimension}\label{section generalisation}

To extend  Theorem \ref{theoreme} to higher dimension, we consider the family of endomorphisms of $\proj^k$ of the form 
$$f:[z_0:..:z_k]\mapsto [P_0(z_0:..:z_{k-1}):..:P_{k-1}(z_0:..:z_{k-1}):z_k^d+\varepsilon P_{k}(z_0:..:z_{k-1})], $$
where $P_0,..,P_k$ are homogeneous polynomials of degree $d\geq 2$ such that $(0,..,0)$ is the single common zero of $P_0,..,P_{k-1}$. 
Let $U_\rho=\lbrace [z_0:..:z_k] ; |z_k|<\rho \max (|z_0|,..,|z_k|)\rbrace$, then if $\varepsilon, \rho$ are small enough we have that $f(U_\rho)\Subset U_\rho$. 
Denote by $f_\infty$ the map $$f_\infty: [z_0:..:z_{k-1}]\mapsto [P_0(z_0:..:z_{k-1}):..:P_{k-1}(z_0:..:z_{k-1})],$$
 $$f_\infty ^+: [z_0:..:z_{k-1}]\mapsto [P_0(z_0:..:z_{k-1}):..:P_{k-1}(z_0:..:z_{k-1}):\varepsilon P_{k}(z_0:..:z_{k-1})].$$
 
 We define $Y$ as before Proposition  \ref{prop petit deg top} and $$X= \left\lbrace p\in \proj^{k-1} \, |\, \card\lbrace f_\infty ^+(p_{-1})  \, |\, p_{-1}\in f_\infty^{-1}(p)\rbrace \leq d^{k-2} \right\rbrace.$$
These are algebraic set of $\proj^{k-1}$. Moreover, for a generic choice of $P_0,..,P_k$, the set $Y$ is  of codimension 1 and $X$ is of dimension 0. 
We may replace the conditions of Proposition \ref{prop petit deg top} by 
\begin{enumerate}
\item $X \cap Y=\emptyset$,
\item  and $\left( f_\infty ^{-(k-1)} ({X\cup Y})  \right) \bigcap .. \bigcap \left( f_\infty ^{-1} ({X\cup Y})  \right) \bigcap (X\cup Y) = \emptyset$.
\end{enumerate}
Following the idea of the proof of Proposition \ref{prop petit deg top}, we show that this implies that $f^{k}$ is of small topological degree on $f(U_\rho)$. 
Therefore, we get a Zariski open set $\Omega$ of mappings asymptotically of small topological  degree on attracting sets in $\proj^k$.
Nevertheless, it is unclear how to find explicit examples of parameters in $\Omega$ to ensure that it is not empty.

\subsection{Non-pluripolar attracting sets with unbounded potentials}

The last remark around Theorem \ref{theoreme} is that it is not necessary to be of small topological degree ``everywhere" on an attracting set in order for it to be non pluripolar.
Furthermore, there exists non-pluripolar attracting sets that cannot support a current of bidegree (1,1) of mass 1  with a bounded quasi-potential.

Let $f$ be an endomorphism of the form 
\begin{equation}\label{def f polynome}
f([z:w:t])=[w^d P\left(\dfrac{z}{w}\right):w^d:t^d+\varepsilon R(z,w)]
\end{equation}
where $P$ is a polynomial of degree $d$ and $R$ is a homogeneous polynomial of degree $d$. 
In particular, it is of the form \eqref{def f}.
On the other hand, $f$ is not asymptotically of small topological degree  and the attracting current
 $\tau$ has a unbounded quasi-potential because $f^{-1}(l_{[1:0]})= \lbrace l_{[1:0]} \rbrace$, hence $\A\cap l_{[1:0]}$ is reduced to a point.
 
 Nevertheless, there exists a bidisk of the form $\lbrace [z:w:t]; |z|\leq R|w|,|t|<\rho \max (|z|,|w|)  \rbrace$ where $f$ is horizontal-like.
 Moreover, we can adapt the proof of Theorem \ref{theoreme} to arrange that this horizontal-like map is of small topological degree.
For this, we replace the generic conditions of Proposition \ref{prop petit deg top} by the following :
\begin{enumerate}
\item \label{cond 1 non bourne} $X\cap Y= \emptyset$
\item \label{cond 2 non bourne} and $f_\infty^{-2}(Z)\cap f_\infty^{-1}(Z)\cap Z = \lbrace [1:0] \rbrace$, with $Z=X\cup Y$.
\end{enumerate}

These properties are clearly satisfied outside some algebraic set and it can be shown that for $c\neq 0$ small enough, $f=[w^d+c z^d: w^d:t^d + \varepsilon (z^d+z^{d-1}w)]$ satisfies these conditions.

Finally, we adapt the proof of Theorem \ref{th potentiel bourne} to horizontal-like maps using the canonical potential.
We recall that it is defined as follows (see \cite{dds} for details).
If $S$ is a horizontal current, i.e. with support in $\D\times \D_{1-\varepsilon}$, its canonical potential is defined by 
$$u_S (z,w)= \int_{\lbrace z \rbrace \times \D} \log |w-s| \, \text{d}m^z(w)$$
with $m^z= S\wedge [\lbrace z \rbrace \times \D]$.
Thus we finally get :

\begin{theoreme}
If $f$ is of the form \eqref{def f polynome}, then for a generic choice of $P,R$ there exists $\varepsilon (P,R)>0$ such that if $\varepsilon(P,R)>|\varepsilon|> 0$ then $f$ admits a non pluripolar attracting set.
\end{theoreme}

\subsection{Other open questions}

There are still two questions left related to Theorem \ref{th potentiel bourne}.

\begin{itemize}
\item In the light of Theorem \ref{th att cont in hyperplan}, it would be interesting to find the right hypothesis ensuring non pluripolarity of attracting set in higher codimension.
\item Does an attracting set of codimension 1 support a current with continuous quasi-potential ?
Can we prove the convergence of the quasi-potentials ? 
\end{itemize}


\end{document}